%% file: main.tex
\documentclass[11pt,a4paper]{article}
\usepackage{float}
\usepackage{makeidx}
\usepackage{layout}
\usepackage{array}
\usepackage{amssymb}
\usepackage{amsmath, longtable}
\usepackage{mathtools}
\usepackage{bbm}
\usepackage{tabularx} 
\usepackage{caption}
\usepackage[utf8]{inputenc}
\usepackage{fancyhdr}
\usepackage{amsthm}
\usepackage{comment}
\usepackage[linesnumbered,boxed,ruled,vlined]{algorithm2e}
\usepackage{graphicx}
\usepackage{scalerel}
\usepackage{authblk}
\usepackage[caption=false]{subfig}
\usepackage{array}
\newcolumntype{H}{>{\setbox0=\hbox\bgroup}c<{\egroup}@{}}

\usepackage[numbers,comma,sort&compress]{natbib}

\usepackage[margin=1in]{geometry}

\usepackage{hyperref}
\hypersetup{
    colorlinks=true,
    linkcolor=blue,
    citecolor=blue,
    filecolor=magenta,      
    urlcolor=blue,
}

\newcommand{\SubProb}{p}

\DeclareMathOperator*{\argmax}{arg\,max}

\DeclarePairedDelimiterX{\inp}[2]{\langle}{\rangle}{#1, #2}

\makeatletter
\def\namedlabel#1#2{\begingroup
    #2%
    \def\@currentlabel{#2}%
    \phantomsection\label{#1}\endgroup
}
\makeatother

\newtheorem{theorem}{Theorem}
\newtheorem{prop}{Proposition}
\newtheorem{lemma}{Lemma}
\theoremstyle{remark}
\newtheorem{remark}{Remark}
\theoremstyle{definition}
\newtheorem{definition}{Definition}

\makeindex
\numberwithin{equation}{section}
\numberwithin{lemma}{section}
\numberwithin{theorem}{section}
\numberwithin{definition}{section}
\numberwithin{prop}{section}
\numberwithin{remark}{section}

\usepackage[pagewise]{lineno}

\title{Memory-efficient structured convex optimization\\ via extreme point sampling}
\author[1,2,3]{Nimita Shinde}
\author[2]{Vishnu Narayanan}
\author[3]{James Saunderson}
\affil[1]{IITB-Monash Research Academy}
\affil[2]{Industrial Engineering and Operations Research, IIT Bombay}
\affil[3]{Electrical and Computer Systems Engineering, Monash University}

\date{}

\begin{document}

\maketitle

\begin{abstract}
Memory is a key computational bottleneck when solving large-scale convex optimization problems such as semidefinite  programs (SDPs). In this paper, we focus on the regime in which storing an $n\times n$ matrix decision variable is prohibitive. To solve SDPs in this regime, we develop a randomized algorithm that returns a random vector whose covariance matrix is near-feasible and near-optimal for the SDP. We show how to develop such an algorithm by modifying the Frank-Wolfe algorithm to systematically replace the matrix iterates with random vectors. As an application of this approach, we show how to implement the Goemans-Williamson approximation algorithm for \textsc{MaxCut} using $\mathcal{O}(n)$ memory in addition to the memory required to store the problem instance. We then extend our approach to deal with a broader range of structured convex optimization problems, replacing decision variables with random extreme points of the feasible region.
\end{abstract}

\input{IntroductionSIMODS}

\input{MethodologySIMODS}

\input{ApplicationsSIMODS}

\input{ConclusionSIMODS}

\bibliographystyle{plainnat}
\bibliography{refShort,References}
\appendix
\input{AppendixSIMODS}

\end{document}

%% file: IntroductionSIMODS.tex
\section{Introduction}

Semidefinite Programs (SDPs) are a class of mathematical programming problems that have a wide range of applications in areas such as control theory, statistical modelling \cite{kakade2007multi, karampatziakis2014discriminative}, correlation clustering \cite{bansal2004correlation, awasthi2015relax}, community detection \cite{hajek2016achieving}, angular synchronization \cite{singer2011angular}, and combinatorial optimization~\cite{vandenberghe1996semidefinite, fares2002robust}. Moreover, a variety of approximation algorithms for combinatorial optimization problems involve solving a SDP relaxation and then rounding the solution to produce a feasible point with provable suboptimality guarantees. There are efficient algorithms such as interior-point methods~\cite{vandenberghe2005interior}, which can be used to solve SDPs. However, as the problem size increases, the memory required by these algorithms becomes a key computational bottleneck. In one regime of interest for very large-scale problems, it is not even possible to store a dense $n\times n$ decision variable in core memory.

One prominent approach to dealing with this bottleneck is to parameterize the positive semidefinite (PSD) decision variable as $X = UU^T$ and reformulate SDPs as nonlinear programs in the variable $U\in \mathbb{R}^{n\times r}$ \cite{burer, homer1997design}, where the value of $r$ must satisfy the Barvinok-Pataki bound \cite{barvinok1995problems,pataki1998rank,waldspurger2020rank} for optimality guarantees. This approach has received a lot of attention (see, for example,~\cite{bhojanapalli2016dropping, chen2015fast, journee2010low, boumal2018global}) because it is able to resolve scalability issues with SDPs to some extent by using a low-rank parameterization. Recently, another approach for SDPs in low memory has emerged. This involves maintaining a lower dimensional sketch of the decision variable, while preserving the convexity of the problem formulation. This approach has primarily been developed in cases where either we know an \textit{a priori} bound on the rank of the solution \cite{ding2019optimal} or the aim is to generate a low-rank approximation of the solution~\cite{yurtsever2017sketchy}.

In this paper, we develop methods to `solve' SDPs in low memory without prior knowledge of the rank of an optimal solution. We do not explicitly aim to represent the matrix decision variable $X$, but rather to sample a zero-mean random vector with covariance $X$. This sampled representation of the solution only requires $\mathcal{O}(n)$ memory and is sufficient to implement rounding schemes for a number of SDP relaxations of binary optimization problems. Moreover, other succinct representations of the solution, such as low-rank approximations, can be computed in low memory by repeatedly generating sampled solutions (see Section~\ref{sec:PostProc}). Our aim, then, is to develop an algorithm that generates random variables with covariance that is a (near-)optimal point to a SDP. Initially, we focus on trace constrained SDPs
\begin{equation}\label{prob:PrimalSDP}
\tag{BoundedSDP}
\max_X\ \ g(\mathcal{B}(X))\ \ \ \textup{subject to}\ \ \ \begin{cases} &\textup{Tr}(X) \leq \alpha\\
&X \succeq 0,\end{cases}
\end{equation}
where the objective function $g$ is concave and smooth and $\mathcal{B}(\cdot): \mathbb{S}^{n} \rightarrow \mathbb{R}^d$ is a linear mapping. This problem class was studied by Yurtsever et al.\  \cite{yurtsever2017sketchy} and our approach is very much inspired by their work. The map $\mathcal{B}$ projects the $\binom{n+1}{2}$-dimensional variable to a much smaller $d$-dimensional space. One way to incorporate any additional constraints to this problem is to add a corresponding penalty term in the objective, at the expense of maintaining exactly feasible iterates (see Section~\ref{sec:LintoBounded}). In Section~\ref{Sec:Applications}, we discuss how to extend further our main idea to certain other constraint sets without compromising on feasibility.

\subsection{Motivating Example: Maximum Cut Problem}\label{Sec:MaxCutIntro}
The \textsc{MaxCut} problem involves maximizing the Laplacian of a graph over binary decision variables, i.e.,
\begin{equation}\label{prob:MaxCut}
\tag{MaxCut}
\max_{x\in \{-1,1\}^n}\ \ x^T C x
\end{equation}
where $C = (1/4)L_G$ and $L_G$ is the Laplacian of a graph. The cost matrix $C$ has a positive entries on the diagonal and is a symmetric diagonally dominant matrix, i.e., the absolute value of each diagonal element of the matrix is greater than or equal to the sum of the absolute values of all the elements in the corresponding row. In a celebrated result, Goemans and  Williamson \cite{GW} developed a $\alpha_{GW}$-approximation algorithm (with $\alpha_{GW}\approx 0.878$) that involves solving the SDP relaxation
\begin{equation}\label{prob:relaxedMC}
\tag{MaxCut-SDP}
\max_{\textup{diag}(X) = \mathbbm{1}, X\succeq 0}\ \ \left\langle C, X\right\rangle
\end{equation}
followed by a randomized rounding scheme. If $X^{\star}$ is an optimal solution of~\eqref{prob:relaxedMC}, the rounding scheme involves sampling a zero-mean Gaussian vector $z$ with covariance $X^{\star}$ and returning the binary vector $\textup{sign}(z)$. This vector
%is a random feasible point to \textsc{MaxCut} that 
achieves the stated approximation guarantee in expectation. To implement this rounding scheme, there is no need to explicitly compute $X^{\star}$; instead it is enough to construct a zero-mean Gaussian vector with covariance $X^{\star}$. This observation is a key motivation for our notion of sampled solutions for SDPs.

Following \cite{yurtsever2019scalable}, we define the working memory of an algorithm as follows.

\begin{definition}
The \emph{working memory} of an algorithm is defined as the total memory utilized by the algorithm apart from the memory required to represent the problem instance.
\end{definition}

If we have an algorithm to solve~\eqref{prob:relaxedMC} that can track such samples rather than the full decision variable, we can potentially implement the Goemans-Williamson method using $\mathcal{O}(n)$ working memory. One of the main contributions of this paper is to show that this in fact possible (see Algorithm~\ref{algo:GenSample}).

\subsection{Our Contributions}
We now summarize the key contributions of the paper.

\paragraph{Sample-based solutions to convex programs} A key conceptual contribution of this paper is to propose the idea of sample-based solutions to convex programs.
\begin{itemize}
\item Gaussian samples: In this case, the aim is to represent the psd solution $X$ of a SDP via a zero mean Gaussian vector $z\sim \mathcal{N}(0,X)$ such that its covariance is $X$.
\item Extreme-point samples: The aim is to represent the solution $x$ of a convex program via a random extreme point of the feasible region such that its expected value is $x$.
\end{itemize}

\paragraph{Generating Gaussian sample-based solution to SDPs} Using an algorithmic framework based on the Frank-Wolfe algorithm, we show that it is possible to compute an $\epsilon$-optimal Gaussian sample-based solution of~\eqref{prob:PrimalSDP} and a near-feasible, near-optimal Gaussian sample-based solution to SDP with $d$ linear equality constraints and bounded feasible region (see Lemma~\ref{lemma:SubOptFeas}). The working memory of our algorithm is $\mathcal{O}(n+d)$.

\paragraph{Approximation algorithm for \textsc{MaxCut}} For \textsc{MaxCut}, we provide an implementation of Goemans-Williamson rounding method that results in a $(1-\epsilon)\alpha_{GW}$-approximate solution (with $\alpha_{GW} \approx 0.878$) to \textsc{MaxCut} that requires additional storage of at most $3n$ numbers. This result Theorem~\ref{thm:MaxCutResult} from Section~\ref{Sec:MaxCut} is stated in a less detailed form below.

\begin{theorem}
Given $\epsilon \in (0,1)$ and a diagonally dominant cost matrix $C$, there exists a polynomial time $\mathcal{O}\left(\frac{n^3}{\epsilon^3}\log(2n)\log\left(\frac{4n}{\epsilon}\right)\times \textup{mvc}\right)$ randomized algorithm, where $\textup{mvc}$ is the complexity of matrix-vector multiplication with $C$, that generates a random binary vector $w$
satisfying
%\in \{-1,1\}^n$ which satisfies
\begin{equation}
\alpha_{GW}(1-\epsilon) \textup{opt} \leq \mathbb{E}[w^TCw] \leq \textup{opt},
\end{equation}
where $\textup{opt}$ is the maximum of $w^TCw$ over the set $w\in \{-1,1\}^n$. The working memory of the algorithm is at most $3n$ numbers.
\end{theorem}

A key conceptual difference between our approach and existing factorization or sketching methods is that no \emph{a priori} bound on the rank of the optimal solution is required.

\paragraph{Generating extreme-point sample-based solutions to convex programs} For a convex optimization problem with compact feasible region, if the extreme points of the feasible region can be represented in low memory and we can solve the linear optimization subproblem in low memory, then we can track its solution in low memory. This allows us to move to a more general setting where the decision variable need not be a positive semidefinite matrix. We provide a modified Frank-Wolfe algorithm (see Algorithm~\ref{algo:FWRandom}) that returns a random extreme point that satisfies the optimality bounds in expectation.

\subsection{Related Work on Low Memory Algorithms for SDP}

%In this subsection, we briefly summarize key ideas in the literature related to low memory algorithms for SDP with $d$ linear equality constraints and bounded feasible region. 
One approach to low memory algorithms for SDPs with $d$ linear equality constraints is to replace the PSD matrix by a low-rank factorization and use nonlinear programming techniques to compute the solution of the resulting nonconvex problem. This technique was pioneered by Burer and Monteiro~\cite{burer}. The factorization sacrifices convexity and its associated optimality guarantees and typical numerical algorithms are only able to generate first- or second-order critical points. Nevertheless, Boumal, Voroninski, and Bandeira~\cite{boumal} showed that if the constraint set is a smooth manifold and the rank $r$ of the factorization is set to satisfy $r(r+1)\geq 2d$, then any second-order critical point for SDP is a global optimum. This bound is a consequence of Barvinok-Pataki bound~\cite{barvinok1995problems, pataki1998rank} which states that SDP with $d$ linear equality constraints and a bounded feasible region admits a global optimum with rank $r^{\star}$ that satisfies $r^{\star}(r^{\star}+1) \leq 2d$. Using a second-order method \cite{journee2010low} or a Riemannian gradient descent algorithm \cite{boumal2018global}, it is then possible to compute an $\epsilon$-optimal solution to the factorized problem that satisfies the Barvinok-Pataki bound, and uses $\mathcal{O}(n\sqrt{d})$ working memory. On the other hand, Waldspurger and Waters~\cite{waldspurger2020rank} showed that unless the rank $r$ of the factorization is set to be at least as large as $\sqrt{2d}$, a typical numerical algorithm may converge to nonoptimal second-order critical points for the resulting nonconvex problem. Thus, this method requires $\Omega(n\sqrt{d})$ working memory to generate a solution with provable optimality guarantees.

A recent approach to solving linearly constrained SDPs in low memory involves sketching the decision variable to a low dimensional subspace and using first-order algorithms to compute a solution of the optimization problem in the space of the sketched decision variable. Unlike factorization approaches that break the convexity of the linearly constrained SDP, sketching the variable preserves the convexity of the problem formulation. These methods are based on techniques for sketching low rank matrices. For instance, Tropp  et  al.~\cite{tropp2017practical}  provide an algorithmically simple sketching technique to compute a low-rank approximation of a matrix $A\in \mathbb{R}^{n\times n}$ by exploiting the spectral decay of the matrix (see also, \cite{tropp2019streaming}). Their algorithm preserves the positive semidefiniteness of the decision variable.

Such sketching and reconstruction techniques  were extended by  Yurtsever et al.~\cite{yurtsever2017sketchy} to generate an approximate \mbox{rank-$r$} factorization of a solution to a smooth convex optimization problem with bounded nuclear norm constraint. They do so by sketching the decision variable to a lower dimension and using the Frank-Wolfe algorithm to track the sketched variable. With the sketched variable, the working memory required for their method is $\Theta (d+nr)$. Although the sketch used by  Yurtsever et al.~\cite{yurtsever2017sketchy} is different from the sample-based solutions in this paper, the algorithmic architecture of our approach is inspired by their sketch.

In their recent work,  Yurtsever et al.~\cite{yurtsever2019scalable} extend the approach from~\cite{yurtsever2017sketchy} to deal with SDPs with $d$ linear equality constraints. They provide a polynomial-time randomized sketching algorithm that, with high probability, computes a rank-$r$ approximation of a near-feasible solution to SDP with $d$ linear equality constraints. The working memory required to compute this solution via sketching is $\mathcal{O}(d+rn/\zeta)$, i.e., it is linear in number of equality constraints and the rank of the computed solution. The quantity $\zeta \in (0,1)$ controls how close the rank-$r$ approximation $\widehat{X}$ produced by the algorithm is to the best rank-$r$ approximation $[X]_r$ of the approximate solution $X$ of SDP with $d$ linear equality constraints. More precisely, with high probability these quantities satisfy 
$\|X-\widehat{X}\|_{\star} \leq (1+\zeta)\|X-[X]_r\|_{\star}$, where 
%$$ is the best rank-$r$ approximation of $X$ and 
$\|\cdot\|_{\star}$ is the nuclear norm.
We will discuss the similarities and differences between our work and~\cite{yurtsever2019scalable} in Section~\ref{Sec:Conclusion}.

Alternatively, Ding et al.~\cite{ding2019optimal} compute a low-rank solution to an SDP with linear equality constraints by first approximately finding the subspace in which the solution lies. This subspace is computed by finding the null space of the dual slack variable. By restricting the search space of the primal solution to a smaller subspace, they restrict the memory required. The dual problem is solved only approximately, with the assumption that the rank of the primal optimal solution satisfies the Barvinok–Pataki bound~\cite{barvinok1995problems, pataki1998rank}, which results in some error in the computation of the subspace in which the primal solution lies. The working memory for their method is then $\Theta (d+nr)$, where $r$ is the rank of an optimal solution to SDP.

A widely studied special case of SDP is~\eqref{prob:relaxedMC}. Since~\eqref{prob:relaxedMC} has $n$ linear equality constraints, the Barvinok-Pataki bound~\cite{barvinok1995problems, pataki1998rank} states that the rank of the optimal solution could be~$\mathcal{O}(\sqrt{n})$. As such, the algorithms for solving~\eqref{prob:relaxedMC} could require~$\mathcal{O}(n^{1.5})$ working memory. One such algorithm is given by Klein and Lu~\cite{klein1996efficient} which generates a \mbox{rank-1} update to the intermediate solution of~\eqref{prob:relaxedMC} at each iteration. The number of iterations required for convergence, and thus the rank of the intermediate solution could be as large as $n$. By restricting the rank of the solution to $\mathcal{O}(\sqrt{n})$, Klein and Lu~\cite{klein1998space} restrict the working memory to be $\mathcal{O}\left(n^{1.5}\right)$. An improvement on this memory requirement is provided by Yurtsever et al.~\cite{yurtsever2019scalable}. Their sketching method generates, with high probability, a rank-1 approximation $zz^T$ of a near-feasible, $\epsilon$-optimal solution $X$ to~\eqref{prob:relaxedMC} using $\mathcal{O}(n/\zeta)$ working memory such that for some $\zeta \in (0,1)$, $\|X-zz^T\|_{\star} \leq (1+\zeta)\|X-\widehat{X}\|_{\star}$, where $\widehat{X}$ is the best rank-1 approximation of $X$ and $\|\cdot\|_{\star}$ is the nuclear norm. While the memory required is linear in $n$, it also has dependence on the approximation parameter~$\zeta$.

Our sampling technique uses $\mathcal{O}(n)$ memory as we aim to represent only the samples of a near-optimal solution. This eliminates the dependency on the accuracy to which the low rank approximation of a near-optimal solution is computed or the rank $r$ of the approximation.

\subsection{Notations}
The inner product $\left\langle A, B \right\rangle = \textup{Tr}\left(A^TB\right)$ denotes the matrix inner product,  $\| A \|_F^2 = \langle A, A\rangle$ is the Frobenius norm and $\| A\|_{\star}$ is the nuclear norm. Unless otherwise specified, $\| \cdot \|$ represents Euclidean norm for vectors. For a matrix $X$, $\textup{diag}(X)$ represents a vector of the diagonal entries of matrix $X$. For a vector $x$, $\textup{diag}^*(x)$ represents a diagonal matrix with the entries of $x$ on the diagonal. The vector $\mathbbm{1}$ has each element value equal to one. The notation $\nabla g(\cdot)$ is used to denote the gradient and $\nabla^2g(\cdot)$ is used to denote the Hessian of a twice differentiable function $g$. We use $\mathcal{A}(\cdot):\mathbb{S}^n\rightarrow \mathbb{R}^d$ to denote a linear mapping of a symmetric $n\times n$ matrix to a $d$-dimensional space and $\mathcal{A}^*(\cdot):\mathbb{R}^d\rightarrow \mathbb{S}^n$ to denote its adjoint. The notation $\succeq$ denotes the semidefinite order. The notation $\lambda_{\textup{max}}(\cdot)$ is used to denote the largest eigenvalue of a matrix. The notations $\mathcal{O}, \Omega, \Theta$ have the usual complexity interpretation.

\subsection{Outline}
The paper is organized as follows. In Section~\ref{Sec:Prelimnaries}, we discuss standard results related to the Frank-Wolfe algorithm. These form the basis of algorithms and analysis discussed later. In Section~\ref{Sec:Methodology}, we present the idea of sampled solutions to SDPs, and provide a modified version of Frank-Wolfe that generates such sampled solutions. In Section~\ref{Sec:MaxCut}, we apply our algorithm to generate sampled solutions to~\ref{prob:relaxedMC}. We also show how to round this sampled solution to give an implementation of the Goemans-Williamson method that uses $\mathcal{O}(n)$ working memory. In Section~\ref{Sec:Applications} we discuss an extension of the idea of sampled solutions to SDPs to a more general setting in which the feasible solution is no longer required to be a trace constrained PSD matrix. We propose a modification of the Frank-Wolfe algorithm that generates a random extreme point of the feasible region, such that the expectation of the generated random solution is near-optimal. In Section~\ref{sec:PostProc}, we discuss how to obtain other approximations of the solution of structured convex programs by combining our sampled representations with streaming algorithms. Section~\ref{Sec:Conclusion} discusses further possible extensions of our work and briefly presents preliminary numerical experiments.

%% file: MethodologySIMODS.tex
\section{Preliminaries}\label{Sec:Prelimnaries}
Consider the following optimization problem,
\begin{equation}\label{prob:BoundedMapped}
\tag{Constrained-OPT}
\max_{x \in \mathcal{S}}\ \ g(\mathcal{B}(x)) = \max_{v\in \mathcal{B}(\mathcal{S})}\ \ g(v)
\end{equation}
where $\mathcal{B}(\cdot):\mathbb{R}^{m}\rightarrow\mathbb{R}^{d}$ is a linear map, $g(\cdot)$ is a smooth, concave function, and $\mathcal{S}\subseteq \mathbb{R}^{m}$ is a compact, convex set. The variable $v = \mathcal{B}(x)$ is said to be a `projection' of the `lifted' decision variable $x$ of~\eqref{prob:BoundedMapped} and the map $\mathcal{B}(\cdot)$ can be interpreted as linear measurements of the decision variable. In this section, we briefly review the Frank-Wolfe algorithm~\cite{frank1956algorithm}, which is central to our algorithmic approach. We recall the modification done to the steps of Frank-Wolfe to adapt it to problems of type~\eqref{prob:BoundedMapped} by~\citet{yurtsever2017sketchy} and give subsequent convergence results. Moreover, we also discuss~\eqref{prob:PrimalSDP}, which is a special case of~\eqref{prob:BoundedMapped}.

\subsection{An Approximate Frank-Wolfe Algorithm}
Frank-Wolfe~\cite{frank1956algorithm} is an iterative algorithm for convex optimization over a compact, convex feasible region that solves a linear optimization problem at each iterate. The algorithm then, in each iteration, takes a fractional step towards a maximizer of the linear optimization subproblem which can be taken to be an extreme point of the feasible set ensuring that the next iterate remains feasible.

An approximate Frank-Wolfe algorithm (see, e.g.,~\cite{FW}) computes an $\epsilon$-optimal solution of~\eqref{prob:BoundedMapped} by solving the linear subproblem approximately (with some finite additive error) at each iteration.

\begin{definition}
A feasible point $\bar{x}\in \mathcal{S}$ is called $\epsilon$-optimal for the optimization problem $\max_{x\in \mathcal{S}} f(x)$ if $f(\bar{x}) \geq \max_{x\in \mathcal{S}} f(x) -\epsilon$.\end{definition}

Algorithm~\ref{algo:FWapprox} details an approximate Frank-Wolfe algorithm to compute an $\epsilon$-optimal solution to~\eqref{prob:BoundedMapped}.

\begin{algorithm}[ht] 
\caption{Frank-Wolfe Algorithm with Approximate Solution to Subproblem~\eqref{prob:Subproblem}}
 \label{algo:FWapprox}
\SetAlgoLined
\SetKwInOut{Input}{Input}\SetKwInOut{Output}{Output}
\DontPrintSemicolon

\Input{Stopping criteria $\epsilon$, accuracy parameter $\eta$, probability $p$ for subproblem~\eqref{prob:Subproblem}, upper bound $C_g^u$ on the curvature constant}
\Output{$\epsilon$-optimal maximizer $x_t$ of $g(\mathcal{B}(x))$, $v_t = \mathcal{B}(x_t)$}
\BlankLine

 \SetKwFunction{CDD}{LMO}
 \SetKwFunction{UV}{UpdateVariable}
 \SetKwFunction{FWA}{FWApproxSubprob}
 \SetKwProg{Fn}{Function}{:}{\KwRet{$(x_t,v_t)$}}

 \Fn{\FWA}{
 Select initial point $x_0\in \mathcal{S}$ and set $v_0 = \mathcal{B}(x_0)$\;
 $t \leftarrow 0$, $\gamma \leftarrow \frac{2}{t+2}$\;
 $(h_t, q_t) \leftarrow \CDD (\mathcal{B}^*(\nabla g(v_{t})), \frac{1}{2}\eta\gamma C_g^u,p)$\;
 
 \While{$\left\langle q_t - v_t, \nabla g(v_t)\right\rangle > \epsilon$}{
 
 $(x_{t+1},v_{t+1}) \leftarrow \UV (x_t, v_t, h_t, q_t,\gamma)$
 
 $t \leftarrow t+1$, $\gamma \leftarrow \frac{2}{t+2}$\;
 
 $(h_t, q_t) \leftarrow \CDD (\mathcal{B}^*(\nabla g(v_{t})), \frac{1}{2}\eta\gamma C_g^u,p$)\;
 
 }
 }
 
 \SetKwProg{Fn}{Function}{:}{\KwRet{$(h,q)$}}
 \Fn{\CDD{$J, \delta$, $p$}}{
 
 Find $h\in \mathcal{S}$ such that with probability at least $1-\SubProb$, $\left\langle h, J\right\rangle \geq \max_{s\in \mathcal{S}} \left\langle s, J\right\rangle - \delta$ \;
 $q = \mathcal{B}(h)$\;
 %Set $\lambda = \langle H, \nabla f(X)\rangle$\;
 }
 
 \SetKwProg{Fn}{Function}{:}{\KwRet{$(x,v)$}}
 \Fn{\UV{$x,v,h,q,\gamma$}}{
 $x \leftarrow (1-\gamma)x + \gamma h$\;
 $v \leftarrow (1-\gamma)v + \gamma q$\; 
 
 }
\end{algorithm}

In Algorithm~\ref{algo:FWapprox}, we assume that at each iteration, the linear optimization subproblem \texttt{LMO} is solved by a randomized method that succeeds with probability at least $1-p$. Moreover, we assume the randomness in this subproblem is independent across function calls.
In Algorithm~\ref{algo:FWapprox}, we not only keep track of the `lifted' iterate $x_t\in \mathbb{R}^n$ (decision variable) at each iteration but also the `projected' iterate $v_t\in\mathbb{R}^d$. The knowledge of these projected iterates $v_t$ is enough to compute the gradient of the objective function, and hence to compute the update direction. Thus, for a high dimensional decision variable, we can simply track the projected iterates $v_t$ to track the improvement in objective function value at each iterate. This is a key observation from the work of~\citet{yurtsever2017sketchy} that is also crucial here.

\paragraph{Curvature constant} The curvature constant $C_g$ is a measure of nonlinearity of the objective function $g$ over the feasible set $\mathcal{S}$. It is defined (see, e.g., \cite{FW}) as,
\begin{equation*}
C_g = \sup_{\substack{x,h\in \mathcal{S},\\ y = v + \gamma(h-v),\\ \gamma \in [0,1]}} -\frac{2}{\gamma^2} \left(g(y) - g(v) - \nabla g(v)^T (y-v)\right).
\end{equation*}
The value of the curvature constant provides insight into the deviation of the linearization of the function $g$ from the actual function value. If $g$ is twice-differentiable, then we often use the upper bound, $C_g \leq C_g^u = \lambda_{\textup{max}}(-\nabla^2 g) \textup{diam}(\mathcal{S})^2$, where $\textup{diam}(\mathcal{S})$ is the Euclidean diameter of the set $\mathcal{S}$.

\paragraph{Approximate solution to subproblem in \texttt{LMO}} The subroutine \texttt{LMO} solves the linear maximization problem
\begin{equation}\label{prob:Subproblem}
\max_{h\in \mathcal{S}}\ \ \left\langle h, \mathcal{B}^*(\nabla g(v))\right\rangle
\end{equation}
approximately at each iteration $t$. The approximate maximizer $h_t$ is computed such that with probability at least $1-\SubProb$, the  additive error in the function value at $h_t$ is at most $\frac{1}{2}\eta\gamma C_g^u$, where $\eta \geq 0$ is a fixed accuracy parameter and $C_g^u$ is an upper bound on the curvature constant of $g$. Initially, this margin of error is high; but with each iteration we solve the subproblem~\eqref{prob:Subproblem} to a higher accuracy. When $\SubProb = 0$, Algorithm~\ref{algo:FWapprox} is equivalent to the Frank-Wolfe algorithm given by~\cite*{FW,hazan2008sparse}. By setting $p>0$, we require the subroutine to adhere to the desired accuracy with some probability of failure.

\paragraph{Stopping criterion}
Algorithm~\ref{algo:FWapprox} terminates when
\begin{equation*}
\left\langle q_t - v_t,\nabla g(v_t)\right\rangle \leq \epsilon
\end{equation*}
is satisfied. 
When this condition holds, it follows that $g(v_t) \geq g(v^{\star}) - \epsilon$, i.e., $v_t$ is an $\epsilon$-optimal solution of~\eqref{prob:BoundedMapped} \cite{FW,hazan2008sparse}.

\paragraph{Convergence of the Frank-Wolfe algorithm}
\begin{theorem} \label{thm:FWapproxConv}
Let $g:  \mathbb{R}^{d} \rightarrow \mathbb{R}$ be a concave and differentiable function and $x^{\star}$ an optimal solution of~\eqref{prob:BoundedMapped}. If $C_g^u$ is an upper bound on the curvature constant of $g$, and $\eta \geq 0$ is the accuracy parameter for subproblem~\eqref{prob:Subproblem}, then $x_t$, the $t$-th iterate of Algorithm~\ref{algo:FWapprox}, satisfies
\begin{equation}
-g(\mathcal{B}(x_{t})) + g(\mathcal{B}(x^{\star})) \leq \frac{2C_g^u(1+\eta)}{t+2}
\end{equation}
with probability $(1-\SubProb)^t \geq 1- t \SubProb$.
\end{theorem}
\begin{proof}
The results follows from \cite[Theorem~1]{FW} when the subproblem~\eqref{prob:Subproblem} is solved to an accuracy of $\frac{1}{2}\eta\gamma C_g^u$ with probability $1-p$.
\end{proof}

Thus, after $t =\frac{2C_g^u(1+\eta)}{\epsilon} - 2$ iterations, the solution $x_t$ satisfies
\begin{equation}\label{eqn:SubOptimalSolPenaltyProb1}
g(\mathcal{B}(x_{t})) \geq g(\mathcal{B}(x^\star)) - \epsilon
\end{equation}
with probability at least $1-tp$.
So, we can now compute an $\epsilon$-optimal solution to~\eqref{prob:BoundedMapped} in $t \sim \mathcal{O}(C_g^u/\epsilon)$ iterations with probability at least $1-tp$.

In Algorithm~\ref{algo:FWapprox}, we solve the subproblem~\eqref{prob:Subproblem} approximately to find an update direction that is an extreme point of the set $\mathcal{S}$. If the following  conditions on the constraint set $\mathcal{S}$ are satisfied, then we potentially generate the sampled representation of the solution to~\eqref{prob:BoundedMapped} in low memory as illustrated in Section~\ref{sec:SDPBoundedTrace}.

\begin{description}
\item[\namedlabel{condition1}{LowMemoryComputations}] The image of the linear map $\mathcal{B}$ is low dimensional and the subproblem~\eqref{prob:Subproblem} can be solved in low memory.
\item[\namedlabel{condition2}{LowMemoryExtremePoints}] The extreme points of the feasible set $\mathcal{S}$ can be represented in low memory.
\end{description}

\subsection{SDP with Bounded Trace Constraint}\label{sec:SDPBoundedTrace}
For most of the paper, our focus is on semidefinite problems. As such, we briefly focus on~\eqref{prob:PrimalSDP}, a special SDP that has a trace constrained feasible set and a concave objective function $g$. The extreme points of the feasible set $\mathcal{S}= \{ X\succeq 0: \textup{Tr}(X)\leq \alpha\}$ are rank-1 PSD matrices and can be represented in low memory. Moreover, a solution of subproblem~\eqref{prob:Subproblem} at iterate $t$ can be computed by finding an eigenvector of the matrix $J = \mathcal{B}^*(\nabla g(\mathcal{B}(X_t)))$ corresponding to its largest eigenvalue. To solve this subproblem approximately, we find $w_t$ such that $\alpha w_t^TJw_t\geq \alpha \lambda_{\textup{max}}(J) - \frac{1}{2}\eta\gamma C_g^u$ and then select the update direction
\begin{equation}\label{eqn:HtforBoundedTrace}
    H_t=
    \begin{cases}
      \alpha w_tw_t^T, & \text{if}\ w_t^TJw_t\geq 0 \\
      0, & \textup{otherwise}.
    \end{cases}
\end{equation}

The computational complexity of each iteration depends on computing a rank-1 matrix that satisfies this inequality, i.e., solving an approximate eigenvalue problem.
We use power method with random start to compute $\lambda_t$ and the unit vector $w_t$.~\citet*{kuczynski1992estimating} provide error bounds for power method when the input matrix is PSD. In Lemma~\ref{lemma:PowerMethod}, we restate their result so it applies to the largest eigenvalue of a symmetric matrix $J$.

\begin{lemma}\label{lemma:PowerMethod}
Let $J\in \mathbb{S}^n$ and let $\lambda$ be the largest absolute eigenvalue of $J$. For $\delta \geq 0$, $\SubProb \in [0,1)$ and $\alpha \geq 0$, the power method with random start computes a unit vector $w$ that satisfies
\begin{equation}
\alpha w^TJw \geq \alpha \lambda_{\textup{max}}(J) - \delta
\end{equation}
with probability at least $1-\SubProb$ after $k \geq \frac{\lambda\alpha}{\delta}\log\left(\frac{n}{p^2}\right)$ iterations. Each iteration of the power method consists of a matrix-vector multiplication with $J$ and the working memory is $n$ numbers.
\end{lemma}

Note that, for~\eqref{prob:PrimalSDP}, it is sufficient to store the input parameters, the map $\mathcal{B}(\cdot)$ and the rank-1 updates. If we have an access to a black box performing matrix-vector multiplications with $J$, the Frank-Wolfe algorithm, when applied to~\eqref{prob:PrimalSDP}, has working memory bounded by $\mathcal{O}(n+d)$ at each step.

\section{Frank-Wolfe Algorithm with Gaussian Sampling}\label{Sec:Methodology}
We now explain how to modify Algorithm~\ref{algo:FWapprox}, when applied to~\eqref{prob:PrimalSDP}, to replace the matrix-valued iterates with Gaussian vectors such that their covariance is equal to the iterate value. We then show how to apply this approach to more general SDPs by incorporating the constraints into the objective with a penalty.

\subsection{Idea of Gaussian Sampling}\label{sec:GS}
Consider the Frank-Wolfe update for~\eqref{prob:PrimalSDP} at iterate $t$, i.e., $X_{t+1} = (1-\gamma_t)X_t + \gamma_t H_t$. Assume that, at iterate $t$, we have a zero-mean random vector $z_t$ with covariance $X_t$. The update direction $H_t = w_tw_t^T$ has rank one, so if $\zeta \sim \mathcal{N}(0,1)$, then $\zeta w_t \sim \mathcal{N}(0,H_t)$. Now, if we define $z_{t+1} = \sqrt{1-\gamma_t}z_t+\sqrt{\gamma_t}\zeta w_t$, then $\mathbb{E}\left[ z_{t+1}z_{t+1}^T\right] = (1-\gamma_t)X_t + \gamma_t H_t = X_{t+1}$.

A nonnegative weighted sum of the samples gives rise to nonnegative weighted sum of their covariance matrices. Furthermore, we can generate a sample at the next iterate in $\mathcal{O}(n)$ memory when the update has rank at most one.

\subsubsection{Frank-Wolfe algorithm with Gaussian sampling}

Algorithm~\ref{Algo1} incorporates Gaussian sampling into Algorithm~\ref{algo:FWapprox} by replacing the matrix variables by the sampled representation. The main algorithm is similar to that in Algorithm~\ref{algo:FWapprox}. The difference lies in the functions \texttt{LMO} and \texttt{UpdateVariable}:
\begin{itemize}
\item \texttt{LMO}: In Algorithm~\ref{algo:FWapprox}, this function simply computes the update direction by computing the approximate maximizer of $\left\langle s, \mathcal{B}^*(\nabla g(v))\right\rangle$ over the feasible set with additive error at most $\frac{1}{2}\eta\gamma C_g^u$. In Algorithm~\ref{Algo1}, we replace this step by computing the vector $w$ such that $w_tw_t^T$ solves the subproblem~\eqref{prob:Subproblem} approximately and the update satisfies~\eqref{eqn:HtforBoundedTrace}.
\item \texttt{UpdateVariable}: In Algorithm~\ref{algo:FWapprox}, we update the decision variable $x$ and its projection $v = \mathcal{B}(x)$ at every iteration. In Algorithm~\ref{Algo1}, we only track the Gaussian sample $z\sim \mathcal{N}(0,X)$ and $v = \mathcal{B}(X)$ which requires only $\Theta(n+d)$ memory instead of $\Theta(n^2+d)$.
\end{itemize}

Algorithm~\ref{Algo1} gives a detailed description of the sampled modification of Frank-Wolfe applied to~\eqref{prob:PrimalSDP}. The convergence rate of the Frank-Wolfe algorithm given in Theorem~\ref{thm:FWapproxConv} applies even with the incorporation of Gaussian sampling, i.e., Algorithm~\ref{Algo1} converges to $z_t$ such that $\mathbb{E}\left[ z_tz_t^T\right]$ is $\epsilon$-optimal for~\eqref{prob:PrimalSDP} after $t\sim \mathcal{O}\left(\frac{C_g^u}{\epsilon}\right)$ iterations. We summarize the result of Algorithm~\ref{Algo1} applied to~\eqref{prob:PrimalSDP} in Proposition~\ref{prop:MainAlgoResult}. 

\begin{algorithm}[ht] 
\caption{Frank-Wolfe Algorithm with Gaussian Sampling} \label{Algo1}

\SetAlgoLined
\SetKwInOut{Input}{Input}\SetKwInOut{Output}{Output}
\DontPrintSemicolon

\Input{Input data for~\eqref{prob:PrimalSDP}, stopping criteria $\epsilon$, accuracy parameter $\eta$, probability $p$ for subproblem~\eqref{prob:Subproblem}, upper bound $C_g^u$ on the curvature constant}

\Output{Sample $z \sim \mathcal{N}(0,\widehat{X}_{\epsilon})$ and $v = \mathcal{B}(\widehat{X}_{\epsilon})$ such that $\widehat{X}_{\epsilon}$ is an $\epsilon$-optimal solution of~\eqref{prob:PrimalSDP}}
\BlankLine

 \SetKwFunction{CDD}{LMO}
 \SetKwFunction{UV}{UpdateVariable}
 \SetKwProg{Fn2}{Function}{:}{\KwRet{$z_t$}}
 \SetKwFunction{FWGS}{FWGaussian}
 \SetKwProg{Fn}{Function}{:}{\KwRet{$(z_t,v_t)$}}
 \Fn{\FWGS}{
 Select initial point $X_0 \in \mathcal{S}$; $X_0 \leftarrow \frac{\alpha}{n}I \textup{(say)}$ and set $v_0 \leftarrow \mathcal{B}(X_0)$\;
 Sample $z_{0} \sim \mathcal{N}(0,X_0)$\;
 $t \leftarrow 0$, $\gamma \leftarrow 2/(t+2)$\;
 $(w_t,q_t) \leftarrow \CDD (\mathcal{B}^*(\nabla g(v_t)), \frac{1}{2}\eta\gamma C_g^u,p)$\;

 \While{$\left\langle q_t - v_t, \nabla g(v_t)\right\rangle > \epsilon$}{
 $(z_{t+1}, v_{t+1}) \leftarrow \UV (z_{t},v_{t},w_t,q_t,\gamma)$\;
 
 $t \leftarrow t+1$, $\gamma \leftarrow 2/(t+2)$\;
 $(w_t,q_t) \leftarrow \CDD (\mathcal{B}^*(\nabla g(v_t)), \frac{1}{2}\eta\gamma C_g^u,p)$\;

 }
 }
 
 \SetKwProg{Fn}{Function}{:}{\KwRet{($w,q$)}}
 \Fn{\CDD{$J$, $\delta$, $p$}}{
 Find a unit vector $w$ such that with probability at least $1-p$, $\alpha\langle ww^T, J\rangle \geq \max_{d\in \mathcal{S}} \alpha\langle d, J\rangle - \delta$\;
 $\lambda \leftarrow \langle ww^T, J\rangle$\;
 \uIf{$\lambda \geq 0$}{
 $q \leftarrow \mathcal{B}(\alpha ww^T)$\;
 }
 \Else{
 $q \leftarrow 0$, $w \leftarrow 0$\;
 }
 }
 
 \SetKwProg{Fn}{Function}{:}{\KwRet{($z,v$)}}
 \Fn{\UV{$z,v,w,q,\gamma$}}{
 $z \leftarrow (\sqrt{1-\gamma})z + \sqrt{\gamma}w\zeta\;\;\textup{where $\zeta\sim \mathcal{N}(0,1)$}$\;
 $v \leftarrow (1-\gamma)v + \gamma q $\;
 
 }

\end{algorithm}

\begin{prop}\label{prop:MainAlgoResult}
The output of Algorithm~\ref{Algo1} is a zero-mean Gaussian random vector  $\widehat{z}_{\epsilon}$ with covariance $\widehat{X}_{\epsilon}$, where $\widehat{X}_{\epsilon}$ is an  $\epsilon$-optimal solution of~\eqref{prob:PrimalSDP}. The working memory of the algorithm is $\mathcal{O}(n+d)$, where the $d$ is the dimension of the image of $\mathcal{B}$.
\end{prop}

\subsection{SDP with Linear Equality Constraints}\label{sec:LintoBounded}
Consider an SDP with linear objective function and linear equality constraints written as,
\begin{equation}\label{prob:SDPlin}
\tag{SDP}
\max_X\ \ \langle C, X \rangle\ \ \ \textup{subject to}\ \ \ \begin{cases}&\mathcal{A}(X) = b\\
&X \succeq 0,\end{cases}
\end{equation}
where $\mathcal{A}(\cdot): \mathbb{S}^{n} \rightarrow \mathbb{R}^d$ is a linear map. We assume that the feasible region is bounded. In Section~\ref{sec:SDPBoundedTrace}, we saw that when Algorithm~\ref{algo:FWapprox} is applied to~\eqref{prob:PrimalSDP}, the rank of the update variable is at most one. And in Section~\ref{sec:GS}, we saw that in this case we can update and track the change in a Gaussian sample representing the matrix-valued decision variable without explicitly computing the matrix at intermediate steps. Because of the simplicity of solving subproblem~\eqref{prob:Subproblem} and updating the samples for trace constrained problems, we penalize the linear equality constraints in~\eqref{prob:SDPlin} with a smooth penalty function so that the feasible domain is reduced to the one in~\eqref{prob:PrimalSDP}. By penalizing the constraints, we can only generate a near-feasible point to~\eqref{prob:PrimalSDP}, however, the extreme points of the modified constraint set now have a concise representation. In the specific case of~\eqref{prob:relaxedMC}, it is possible to generate a feasible solution with relative error bounds on the objective function value if we know a near-feasible solution for the problem with infeasibility error bounded by $\|\mathcal{A}(X)-b\|_{\infty}$. This motivates us to use a penalty function~\eqref{eqn:penaltyPhiDef} that approximates $\|\mathcal{A}(X)-b\|_{\infty}$.

\paragraph{Penalty function} For $M>0$, let $\phi_M(\cdot):\mathbb{R}^{d}\rightarrow \mathbb{R}$ be defined by
\begin{equation}\label{eqn:penaltyPhiDef}
\phi_M(v) = \frac{1}{M}\log\left(\sum_{i=1}^{d} e^{M(v_i)} + \sum_{i=1}^{d} e^{M(-v_i)}\right).
\end{equation}

This function, also known as \textsc{LogSumExp} (LSE), is a smoothed approximation of $\|v\|_{\infty}$ as the next well-known proposition shows.

\begin{prop}[Bound on penalty~\cite{grigoriadis1994fast}]\label{prop:BoundPenalty}
If $\phi_M(\cdot)$ is defined as in~\eqref{eqn:penaltyPhiDef}, then 
\begin{equation}\label{eqn:BoundOnPenalty}
\| v\|_{\infty} \leq \phi_M (v) \leq \frac{\log(2d)}{M} + \| v\|_{\infty}.
\end{equation}
\end{prop}

We add this penalty term to the objective function of~\eqref{prob:SDPlin} to penalize the equality constraints and then solve the problem
\begin{equation} \label{prob:LogPenalty}
\tag{SDP-LSE}
\max_X\ \ \langle C,X\rangle - \beta \phi_M(\mathcal{A}(X)-b)\ \ \ \textup{subject to}\ \ \ \begin{cases}&\textup{Tr}(X) \leq \alpha,\\
& X \succeq 0,\end{cases}
\end{equation}
where $M$ and $\beta$ are positive constants to be chosen later, and $\alpha$ is chosen such that $X\succeq 0$ and $\mathcal{A}(X)=b$ implies that $\textup{Tr}(X)\leq \alpha$. This is possible because the feasible region is assumed to be bounded.

A similar approximation and penalty technique is used by~\citet*{hazan2008sparse} to compute an approximate solution of a feasibility problem with linear inequality constraints. The objective function of~\eqref{prob:LogPenalty} in this case is simply $\phi_M(\mathcal{A}(X) - b)$ and the optimal objective function value is zero.

Let $(u,v) = \mathcal{B}(X) = (\langle C,X\rangle, \mathcal{A}(X))$ so that the objective function of~\eqref{prob:LogPenalty} can be expressed as $g(u,v) = u - \beta \phi_M(v-b)$. This problem then has the same structure as~\eqref{prob:PrimalSDP}. We can now use Algorithm~\ref{Algo1} to compute an $\epsilon$-optimal solution $\widehat{X}_{\epsilon}$ of~\eqref{prob:LogPenalty}. Moreover, we will show (in Lemma~\ref{lemma:SubOptFeas}) that by choosing the parameters $M$ and $\beta$ approximately, $\widehat{X}_{\epsilon}$ is also a near-feasible, near-optimal solution of~\eqref{prob:SDPlin}. The convergence rate and the bounds on the infeasibility and objective function value at $\widehat{X}_{\epsilon}$ for~\eqref{prob:SDPlin} are given in Lemma~\ref{lemma:SubOptFeas}.

\subsubsection{Convergence of the Frank-Wolfe algorithm} \label{sec:ConvergenceFW} Theorem~\ref{thm:FWapproxConv} states the convergence result of Algorithm~\ref{algo:FWapprox} when applied to~\eqref{prob:BoundedMapped}. The algorithm converges to an $\epsilon$-optimal solution after $\mathcal{O}\left( \frac{C_g^u}{\epsilon}\right)$ iterations. Moreover, this convergence result also holds for Algorithm~\ref{Algo1} when applied to~\eqref{prob:LogPenalty}. We now determine an upper bound on the curvature constant $C_g$ for~\eqref{prob:LogPenalty}.

\begin{lemma}\label{lemma:boundOnCf}
An upper bound on the curvature constant $C_g$ of the concave, smooth function $g(\mathcal{B}(X)) = \langle C, X\rangle - \beta \phi_M (\mathcal{A}(X)-b)$, where $\mathcal{A}(X) = [\langle A_i, X\rangle]_{i=1}^{d}$, over the compact, convex set $\mathcal{S} = \{X\succeq 0: \textup{Tr}(X) \leq \alpha\}$ is
\begin{equation*}
C_g \leq \beta \omega M \alpha^2,
\end{equation*}
where $\omega = \max_i \lambda_{\textup{max}}(A_i)$.
\end{lemma}

\begin{proof}
Let $f(X) = -\phi_M(\mathcal{A}(X)-b)$ and let $\mathcal{S}_f = \{X\succeq 0: \textup{Tr}(X) \leq 1\}$. An upper bound on the curvature constant $C_f$ of $f$ over $\mathcal{S}_f$, given by~\citet*{hazan2008sparse} is
\begin{equation}
C_f \leq \lambda_{\textup{max}}(-\nabla^2f)\textup{diam}(\mathcal{S}_f)^2 \leq \omega M.
\end{equation}

Note that $\lambda_{\textup{max}}(-\nabla^2g) = \beta \lambda_{\textup{max}}(-\nabla^2f)$ and $\textup{diam}(\mathcal{S})^2 =  \alpha^2\textup{diam}(\mathcal{S}_f)^2$. Hence, an upper bound on the curvature constant of $g$ over $\mathcal{S}$ is \[ C_g \leq \lambda_{\textup{max}}(-\nabla^2g)\textup{diam}(\mathcal{S})^2 = \beta \alpha^2 \lambda_{\textup{max}}(-\nabla^2f)\textup{diam}(\mathcal{S}_f)^2 \leq \beta \omega M\alpha^2. \]
\end{proof}

\subsubsection{Optimality and feasibility results for~\eqref{prob:SDPlin}}
Given an $\epsilon$-optimal solution to~\eqref{prob:LogPenalty}, $\widehat{X}_{\epsilon}$, we analyze its suboptimality and infeasibility
with respect to~\eqref{prob:SDPlin}. The dual of~\eqref{prob:SDPlin} is
\begin{equation}\label{prob:SDPLinDual}
\tag{DSDP}
\min_y\ \ b^Ty\ \ \ \textup{subject to}\ \ \ \mathcal{A}^*(y) - C \succeq 0.
\end{equation}

We assume that~\eqref{prob:SDPlin} is feasible and a constraint qualification holds ensuring that strong duality is satisfied, and the primal and dual problems have finite optimal values. Let ($X^{\star}_{SDP},y^{\star}_{SDP}$) be a primal-dual optimal pair. The optimality and infeasibility bounds that we obtain depend on the properties of the optimal dual solution, $y^{\star}_{SDP}$.

\begin{lemma}\label{lemma:SubOptFeas}
Let $(X^{\star}_{SDP}, y^{\star}_{SDP})$ be an optimal solution of~\eqref{prob:SDPlin}\textendash\eqref{prob:SDPLinDual} and let $\widehat{X}_{\epsilon}$ be an $\epsilon$-optimal solution of~\eqref{prob:LogPenalty}. If $\beta > \| y^{\star}_{SDP}\|_1$, then
\begin{equation}\label{eqn:OptBound}
\langle C, X^{\star}_{SDP}\rangle - \epsilon \leq \langle C, \widehat{X}_{\epsilon}\rangle \leq \langle C, X^{\star}_{SDP}\rangle + \| y^{\star}_{SDP}\|_1 \frac{\beta\frac{\log(2d)}{M} + \epsilon}{\beta - \| y^{\star}_{SDP}\|_1},
\end{equation}
and
\begin{equation}\label{eqn:FeasBound}
\| \mathcal{A}(\widehat{X}_{\epsilon}) -b \|_{\infty} \leq \frac{\beta\frac{\log(2d)}{M} + \epsilon}{\beta - \| y^{\star}_{SDP}\|_1}.
\end{equation}
\end{lemma}

\begin{proof}
The bounds~\eqref{eqn:OptBound} and~\eqref{eqn:FeasBound} are derived in Appendix~\ref{appendix:OptFeasLin}.
\end{proof}

\begin{remark}
Lemma~\ref{lemma:SubOptFeas} shows that $\widehat{X}_{\epsilon}$ is a near-feasible point to~\eqref{prob:SDPlin} with bounded objective function value such that the infeasibility and optimality bounds depend on the dual solution, $\| y^{\star}_{SDP}\|_1$, and the parameters $\beta$ and $M$. If the parameter values $\beta$ and $M$ are specifically chosen to be $\beta = 2\| y^{\star}_{SDP}\|_1$ and $M = 2\frac{\log(2d)}{\epsilon}$, then $\widehat{X}_{\epsilon}$ satisfies $\| \mathcal{A}(\widehat{X}_{\epsilon}) -b \|_{\infty} \leq 2\epsilon$ and the objective function value is upper bounded by $\langle C, X^{\star}_{SDP}\rangle + 2\epsilon\| y^{\star}_{SDP}\|_1$.
\end{remark}

\begin{remark}
In some cases, it is difficult to produce truly feasible points from the generated near-feasible points. While in other cases, such as~\eqref{prob:relaxedMC}, it is fairly straightforward to make small modifications to a generated near-feasible point to produce a feasible point with a similar objective value. Furthermore, we will see that for~\eqref{prob:relaxedMC}, it is possible to reduce $\| y^{\star}_{SDP}\|_1$ to $\langle C, X^{\star}_{SDP}\rangle$, eliminating the unknown optimal dual variable from our bounds.
\end{remark}

\section{Approximation Algorithm for \textsc{MaxCut}}\label{Sec:MaxCut}
We apply the general framework from Section~\ref{Sec:Methodology} to give an implementation of the Goemans-Williamson approximation algorithm for \textsc{MaxCut} that uses only $\mathcal{O}(n)$ working memory. Recall from Section~\ref{Sec:MaxCutIntro} that the standard SDP relaxation of \textsc{MaxCut},~\eqref{prob:relaxedMC} is a special case of~\eqref{prob:SDPlin} with a symmetric diagonally dominant cost matrix $C$ and the constraint set $\{X\in \mathbb{S}^n: X\succeq 0,\textup{diag}(X) = \mathbbm{1}\}$. A key challenge in applying the results of Lemma~\ref{lemma:SubOptFeas} in this setting is that the suboptimality and infeasibility bounds obtained depend on the optimal solution of the dual SDP~\eqref{prob:SDPLinDual}. With no \emph{a priori} knowledge of the dual solution, selecting the value of parameter $\beta$ is difficult and the additive error term in~\eqref{eqn:OptBound} could be quite high. Furthermore, setting the value of $\beta$ to an arbitrarily large value increases the curvature constant.

For~\eqref{prob:relaxedMC}, we show how to apply Lemma~\ref{lemma:SubOptFeas} to obtain relative error bounds on the suboptimality and infeasibility of the output of Algorithm~\ref{Algo1} without \textit{a priori} knowledge of the dual optimal solution. This allows us to appropriately choose the parameters $\beta$ and $M$ in the penalized formulation~\eqref{prob:MCpenalized}. Moreover, it is possible to generate a feasible solution to \textsc{MaxCut} that nearly achieves the approximation guarantee of Goemans-Williamson method, by applying Algorithm~\ref{algo:GenSample} (see Section~\ref{sec:MCRounding}) to the Gaussian samples generated from the PSD matrix produced by Algorithm~\ref{Algo1}.

Let the constraint `$\textup{diag}(X) = \mathbbm{1}$' in~\eqref{prob:relaxedMC} be penalized with $\phi_M(\cdot)$~\eqref{eqn:penaltyPhiDef} and consider the modified problem
\begin{equation}\label{prob:MCpenalized}
\tag{MaxCut-LSE}
\max_X\ \ \langle C, X\rangle - \beta \phi_M (\textup{diag}(X) -\mathbbm{1})\ \ \ \textup{subject to}\ \ \ \begin{cases} &\textup{Tr}(X) \leq n\\
&X\succeq 0.
\end{cases}
\end{equation}

We will see that choosing $M = 4\frac{\log(2d)}{\epsilon}$ and $\beta = 4\textup{Tr}(C)$ in~\eqref{prob:MCpenalized} allows us to derive a relative error bound on the objective function value of~\eqref{prob:relaxedMC}.  The main result of this section is given as follows:

\begin{theorem}\label{thm:MaxCutResult}
Let~\eqref{prob:MCpenalized} be solved to $\epsilon\textup{Tr}(C)$-optimality using Algorithm~\ref{Algo1} with $\epsilon \in \left(0,\frac{1}{3}\right)$, $\eta =1$, $p = \frac{\epsilon}{T(n,\epsilon)}$, where $T(n,\epsilon) = 64\frac{\log(2n)n^2}{\epsilon^2}$, followed by the rounding scheme of Algorithm~\ref{algo:GenSample}. For a diagonally dominant matrix $C$, this procedure generates a binary vector $w$ that satisfies
\begin{equation}\label{eqn:MCoptimality}
\alpha_{GW}(1-3\epsilon) \textup{opt} \leq \mathbb{E}[w^TCw] \leq \textup{opt},
\end{equation}
where $\textup{opt}$ is the maximum of $w^TCw$ over the set $w \in \{-1,1\}^n$. Algorithm~\ref{Algo1} terminates after at most $T(n,\epsilon)$ iterations, where at each iteration, at most $240 \frac{n}{\epsilon}\log\left(\frac{4n}{\epsilon}\right)$ matrix-vector multiplications are performed. The working memory required is at most $3n$ numbers.
\end{theorem}

The proof of Theorem~\ref{thm:MaxCutResult} is given at the end of Section~\ref{sec:MCRounding}.

\subsection{Relative Error Bounds on Suboptimality and Infeasibility}
In this subsection, we derive relative error bounds on the suboptimality and infeasibility for~\eqref{prob:relaxedMC} at $\widehat{X}_{\epsilon}$, the output of Algorithm~\ref{Algo1}, when solving~\eqref{prob:MCpenalized}. This is an application of Lemma~\ref{lemma:SubOptFeas} with the key difference that any dependence on the dual optimal solution has been eliminated.

The dual of SDP relaxation of \textsc{MaxCut} is,
\begin{equation}\label{prob:MCSDPDual}
\min_y\ \ \sum_{i=1}^n y_i\ \ \ \textup{subject to}\ \ \ \textup{diag}^*(y) - C \succeq 0.
\end{equation}

Let $y^\star_{SDP}$ be an optimal solution of~\eqref{prob:MCSDPDual}.
We first show that $\| y^{\star}_{SDP} \|_1$ is upper bounded by $2\textup{Tr}(C)$.

\begin{lemma}\label{lemma:MaxCutSDPbounds}
Let $(X^\star_{SDP}, y^{\star}_{SDP})$ be a primal-dual optimal pair for~\eqref{prob:relaxedMC} and its dual~\eqref{prob:MCSDPDual}. If $C$ is diagonally dominant, then
\begin{equation}
\textup{Tr}(C) \leq \langle C, X^{\star}_{SDP}\rangle = \| y^{\star}_{SDP}\|_1 \leq 2 \textup{Tr}(C).
\end{equation}
\end{lemma}

\begin{proof}
For a symmetric diagonally dominant cost matrix $C$ with nonnegative entries on the diagonal, it follows from the Gershgorin cirle theorem, that $C$ must be a PSD matrix. Since $C$ is PSD, $\textup{diag}^*(y) - C \succeq 0$ implies $y \geq 0$. Thus, the objective function of~\eqref{prob:MCSDPDual} can be written as $\| y\|_1$. Moreover, the SDP relaxation of \textsc{MaxCut} satisfies Slater's condition, so
\begin{equation*}
\langle C, X^{\star}_{SDP}\rangle = \| y^{\star}_{SDP}\|_1,
\end{equation*}
for a primal-dual optimal pair $(X^{\star}_{SDP},y^{\star}_{SDP})$. Furthermore, to see that $\textup{Tr}(C) = \langle C, I \rangle \leq \langle C, X^{\star}_{SDP}\rangle$, we simply note that $I$ is feasible for~\eqref{prob:relaxedMC}.

To prove $\langle C, X^{\star}_{SDP}\rangle \leq 2\textup{Tr}(C)$, we use the fact that $C$ is a diagonally dominant matrix with nonnegative entries on the diagonal. As such $2\textup{diag}^*(\textup{diag}(C)) - C$ is also symmetric diagonally dominant and has nonnegative diagonal entries. It follows that
\begin{equation}
2\textup{diag}^*(\textup{diag}(C)) - C \succeq 0.
\end{equation}
Then,
\begin{equation}
\langle C, X^{\star}_{SDP}\rangle \leq  \langle 2\textup{diag}^*(\textup{diag}(C)), X^{\star}_{SDP}\rangle = 2\textup{Tr}(C),
\end{equation}
where we used the fact that $X^{\star}_{SDP} \succeq 0$ and $\textup{diag}(X^{\star}_{SDP}) = \mathbbm{1}$.
\end{proof}

\begin{remark}
If $C \succeq 0$, but is not diagonally dominant, then $\| y^{\star}_{SDP}\|_1 \leq n \lambda_{\textup{max}}(C) \leq n\textup{Tr}(C)$, since $y = \mathbbm{1}\cdot\lambda_{\textup{max}}(C)$ is feasible for the dual~\eqref{prob:MCSDPDual}. So, $\textup{Tr}(C) \leq \langle C, X^{\star}_{SDP}\rangle = \| y^{\star}_{SDP}\|_1 \leq n\textup{Tr}(C)$ in that case.
\end{remark}

\subsubsection{Optimality and feasibility bounds for~\eqref{prob:relaxedMC}} \label{sec:ConvergenceFW2}
By finding a near-optimal point for~\eqref{prob:MCpenalized}, the penalized relaxation of~\eqref{prob:relaxedMC}, we can obtain a near-feasible solution to~\eqref{prob:relaxedMC} that has relative error~$\epsilon$. Note that a relative error bound is exactly what we need to obtain a multiplicative approximation guarantee for \textsc{MaxCut}.

\begin{lemma}\label{lemma:SubOptFeasMaxCut}
Let $X^{\star}_{SDP}$ be an optimal solution of ~\eqref{prob:relaxedMC} and $\widehat{X}_{\epsilon}$ be an $\epsilon\textup{Tr}(C)$-optimal solution to~\eqref{prob:MCpenalized} with $M = 4 \frac{\log(2d)}{\epsilon}$ and $\beta = 4\textup{Tr}(C)$. Then
\begin{equation}\label{eqn:MCbounds}
\langle C, X^{\star}_{SDP}\rangle(1 - \epsilon) \leq \langle C, \widehat{X}_{\epsilon}\rangle \leq \langle C, X^{\star}_{SDP}\rangle (1 + \epsilon),
\end{equation}
and
\begin{equation}\label{eqn:MCinfeasbounds}
\| \mathcal{A}(\widehat{X}_{\epsilon}) -b \|_{\infty} \leq \epsilon.
\end{equation}
\end{lemma}

\begin{proof}
This result is an application of Lemma~\ref{lemma:SubOptFeas} for specific parameter values. Substituting the values of $M$ and $\beta$, and 
using the inequality $\|y^\star_{SDP}\|_1 \leq 2\textup{Tr}(C)$ from Lemma~\ref{lemma:MaxCutSDPbounds}, we see that 
\[ \| \mathcal{A}(\widehat{X}_{\epsilon}) - b\|_{\infty} \leq  \frac{\beta\frac{\log(2d)}{M} + \epsilon \textup{Tr}(C)}{\beta - \|y_{SDP}^\star\|_1} \leq \frac{2\epsilon \textup{Tr}(C)}{4\textup{Tr}(C) - 2\textup{Tr}(C)} = \epsilon.\]
Furthermore, combining $\langle C,X^\star_{SDP}\rangle = \|y^\star_{SDP}\|_1$ (from Lemma~\ref{lemma:MaxCutSDPbounds}) with Lemma~\ref{lemma:SubOptFeas} gives the upper bound on $\langle C, \widehat{X}_{\epsilon}\rangle$. Finally, using the fact that $\langle C, X^\star_{SDP}\rangle \geq \textup{Tr}(C)$ and substituting in~\eqref{eqn:OptBound}, gives
\begin{equation*}
\langle C, \widehat{X}_{\epsilon}\rangle \geq \langle C, X^\star_{SDP}\rangle - \epsilon \textup{Tr}(C) \geq \langle C, X^\star_{SDP}\rangle (1-\epsilon).
\end{equation*}
\end{proof}

\begin{remark}
If $C \succeq 0$, but not necessarily diagonally dominant, then the bounds given in Lemma~\ref{lemma:SubOptFeasMaxCut} hold for $M = (n+2)\frac{\log(2d)}{\epsilon}$ and $\beta = (n+2)\textup{Tr}(C)$.
\end{remark}

\subsection{Generating a Feasible Solution to~\textsc{MaxCut}}\label{sec:MCRounding}
We now show how to adapt the rounding procedure of Goemans-Williamson to our setting (Algorithm~\ref{algo:GenSample}). The reason we need to modify the Goemans-Williamson scheme is because the zero-mean Gaussian random vector returned by Algorithm~\ref{Algo1} has covariance that is not feasible for~\eqref{prob:relaxedMC}.

\begin{algorithm}[ht] 
\SetAlgoLined
\SetKwInOut{Input}{Input}\SetKwInOut{Output}{Output}
\DontPrintSemicolon

\Input{A sample $\widehat{z}_{\epsilon} \sim \mathcal{N}(0,\widehat{X}_{\epsilon})$ and $\textup{diag}(\widehat{X}_{\epsilon})$}
\Output{A feasible solution to \textsc{MaxCut} ${w} = \textup{sign}(\overline{w})$}
\BlankLine

\SetKwFunction{GS}{GenerateSample}
\SetKwProg{Fn}{Function}{:}{\KwRet{$\textup{sign}(\overline{w})$}}

 \Fn{\GS}{
 Generate $\zeta \sim \mathcal{N}\left(0, I - \textup{diag}^*\left(\frac{\textup{diag}(\widehat{X}_{\epsilon})}{\max (\textup{diag}(\widehat{X}_{\epsilon}))}\right)\right)$\;
 Set $\overline{w} = \frac{\widehat{z}_{\epsilon}}{\sqrt{\max (\textup{diag}(\widehat{X}_{\epsilon}))}} + \zeta$
 }
\caption{\label{algo:GenSample}Generate a binary vector from a Gaussian vector}
\end{algorithm}

Algorithm~\ref{algo:GenSample} can be used to generate a feasible solution of \textsc{MaxCut} from any PSD (covariance) matrix $X$. The first step of the algorithm generates $n$ independent zero-mean random variables with covariance defined by the diagonal entries of $I - \textup{diag}^*\left(\frac{\textup{diag}(\widehat{X}_{\epsilon})}{\max (\textup{diag}(\widehat{X}_{\epsilon}))}\right)$. The random vector $\overline{w}$ in step 3 is a sum of two independent zero-mean Gaussian random vectors. The covariance of this random vector $\overline{w}$ can be stated as
\begin{equation}\label{eqn:GenerateFeasMatrix}
\overline{X} = \frac{\widehat{X}_{\epsilon}}{\max (\textup{diag}(\widehat{X}_{\epsilon}))} + \left(I - \textup{diag}^*\left(\frac{\textup{diag}(\widehat{X}_{\epsilon})}{\max (\textup{diag}(\widehat{X}_{\epsilon}))}\right)\right)
\end{equation}
so that $\overline{w} \sim \mathcal{N}(0,\overline{X})$. The matrix $\overline{X}$ is a sum of two PSD matrices and so is PSD. Moreover, $\textup{diag}(\overline{X}) = \mathbbm{1}$, so $\overline{X}$ is feasible for~\eqref{prob:relaxedMC}. We can then apply the standard analysis of the Goemans-Williamson rounding scheme to $\overline{X}$.

\paragraph{Goemans-Williamson rounding} For a PSD matrix $C$ and a Gaussian random vector $\overline{w} \sim \mathcal{N}(0,\overline{X})$, such that $\textup{diag}(\overline{X}) = \mathbbm{1}$,~\citet*{nesterov1998semidefinite} derived a $\frac{2}{\pi}$-approximation bound,
\begin{equation}\label{eqn:GWBound}
\mathbb{E}_G[w^TCw] \geq \alpha \langle C, \overline{X}\rangle,
\end{equation}
where $w = \textup{sign}(\overline{w})$, $\alpha = \frac{2}{\pi}$ and $\mathbb{E}_G[\cdot]$ denotes the expectation over Gaussian random vectors. Moreover, if $C$ is diagonally dominant,~\citet*{GW} provide a tighter bound with $\alpha = \alpha_{GW} \approx 0.878$.

When the input of Algorithm~\ref{algo:GenSample} is an approximate solution of~\eqref{prob:relaxedMC}, we analyze the expected objective value of $\textup{sign}(\overline{w})$.

\begin{lemma}\label{lemma:MaxCutapproxSol}
Let $\epsilon \in \left(0,\frac{1}{2}\right)$, $C$ be a diagonally dominant matrix, and let $\widehat{X}_{\epsilon} \succeq 0$ satisfy the bounds given in Lemma~\ref{lemma:SubOptFeasMaxCut}. If a binary vector $w = \textup{sign}(\overline{w})$ is generated by Algorithm~\ref{algo:GenSample} with input $\widehat{z}_{\epsilon} \sim \mathcal{N}(0, \widehat{X}_{\epsilon})$, then the expected value of $w^TCw$ satisfies
\begin{equation}
\alpha_{GW} (1-2\epsilon)\textup{opt} \leq \mathbb{E}_G[w^TCw] \leq \textup{opt} \leq \langle C, X^\star_{SDP}\rangle,
\end{equation}
where $\textup{opt}$ is the optimal value of $w^TCw$ over the set $w \in \{\pm 1 \}^n$.
\end{lemma}

\begin{proof}
The objective function value of~\eqref{prob:relaxedMC} at $\overline{X}$ is
\begin{align}
\langle C, \overline{X}\rangle &= \left\langle C, \frac{\widehat{X}_{\epsilon}}{\max (\textup{diag}(\widehat{X}_{\epsilon}))} + \left(I - \textup{diag}^*\left(\frac{\textup{diag}(\widehat{X}_{\epsilon})}{\max (\textup{diag}(\widehat{X}_{\epsilon}))}\right)\right)\right\rangle\\
&\geq \frac{\langle C, \widehat{X}_{\epsilon}\rangle}{\max (\textup{diag}(\widehat{X}_{\epsilon}))}\label{eqn:MCApprox1}\\
&\geq \frac{1-\epsilon}{1+\epsilon}\langle C, X^{\star}_{SDP}\rangle \label{eqn:MCApprox2}\\
&\geq (1-2\epsilon)\langle C, X^{\star}_{SDP}\rangle, \label{eqn:MCApprox3}
\end{align}
where~\eqref{eqn:MCApprox1} follows from the fact that both $C$ and $I -  \textup{diag}^*\left(\frac{\textup{diag}(\widehat{X}_{\epsilon})}{\max (\textup{diag}(\widehat{X}_{\epsilon}))}\right)$ are PSD and their inner product is greater than 0,~\eqref{eqn:MCApprox2} follows from Lemma~\ref{lemma:SubOptFeasMaxCut}, and~\eqref{eqn:MCApprox3} uses the fact that $\frac{1}{1+\epsilon} \geq 1-\epsilon$ and $(1-\epsilon)^2 \geq 1-2\epsilon$. Substituting~\eqref{eqn:MCApprox3} in~\eqref{eqn:GWBound} gives the desired result.
\end{proof}

Note that the result in Lemma~\ref{lemma:MaxCutapproxSol} holds irrespective of the algorithm used to compute $\widehat{X}_{\epsilon}$. We are now in a position to prove Theorem~\ref{thm:MaxCutResult}.

\begin{proof}[Proof of Theorem~\ref{thm:MaxCutResult}]
Since we use Algorithm~\ref{Algo1} to solve~\eqref{prob:MCpenalized} with $p = \frac{\epsilon}{T(n,\epsilon)}$, the bounds in Lemma~\ref{lemma:SubOptFeasMaxCut} are satisfied with probability at least $1-\epsilon$. Thus, the bound $\langle C, \overline{X}\rangle \geq (1-2\epsilon)\langle C, X^{\star}_{SDP}\rangle$ also holds with probability at least $1-\epsilon$. Moreover, for any $\overline{X} \succeq 0$, $\langle C, \overline{X}\rangle \geq 0$, and thus, a lower bound on the expected value of $\langle C, \overline{X}\rangle$ over the random initialization of the power method, i.e., $\mathbb{E}_P[\langle C, \overline{X}\rangle]$, at each iteration is,
\begin{align*}
\mathbb{E}_P[\langle C, \overline{X}\rangle] &\geq (1-2\epsilon)\langle C, X^{\star}_{SDP}\rangle
(1-\epsilon)\\
&\geq (1-3\epsilon) \langle C, X^{\star}_{SDP}\rangle.
\end{align*}

The lower bound in~\eqref{eqn:MCoptimality} then follows from~\eqref{eqn:GWBound} because
\begin{align*}
\mathbb{E}[w^TCw] &= \mathbb{E}_P[\mathbb{E}_G[w^TCw]]\\
&\geq \alpha_{GW}\mathbb{E}_P[\langle C, \overline{X}\rangle]\\
&\geq \alpha_{GW}(1-3\epsilon)\langle C, X^{\star}_{SDP}\rangle\\
&\geq \alpha_{GW}(1-3\epsilon)\textup{opt}.
\end{align*}

\paragraph{Bound on $T$, number of iterations of Algorithm~\ref{Algo1}}
An upper bound on the curvature constant of~\eqref{prob:MCpenalized} is $C_g^u = 16\frac{\textup{Tr}(C)\log(2n)n^2}{\epsilon}$ since $\omega = 1$ and $\alpha = n$. Algorithm~\ref{Algo1} converges to an $\epsilon \textup{Tr}(C)$-optimal solution after at most $T =\frac{2C_g^u(1+\eta)}{\epsilon \textup{Tr}(C)} - 2 \leq 64 \frac{\log(2n)n^2}{\epsilon^2}$ iterations with probability at least $1-Tp$.

\paragraph{Bound on number of iterations of power method at each $t$}
From Lemma~\ref{lemma:PowerMethod}, the number of matrix-vector multiplications performed at iteration $t$ of Algorithm~\ref{Algo1} is at most $\frac{\lambda\alpha}{\delta}\log\left( \frac{n}{p^2}\right)$ with $\delta = \frac{1}{2}\eta \gamma_t C_g^u \approx \frac{1}{8}\epsilon\textup{Tr}(C)$ and
\begin{align*}
\lambda &= \max_i |\lambda_i (\nabla g(v_t))|\\
&= \max_i |\lambda_i (C - \beta D)|,
\end{align*}
where $D$ is a diagonal matrix with $d_{ii} \in [-1,1]$.
Thus, $\lambda \leq \lambda_{\textup{max}}(C) + \beta \leq 5\textup{Tr}(C)$.

Substituting the value of $p$, and bounds on $\lambda$ and $\delta$ in Lemma~\ref{lemma:PowerMethod}, we have
\begin{equation*}
\begin{split}
\frac{\lambda\alpha}{\delta}\log\left( \frac{n}{p^2}\right) &= \frac{40n}{\epsilon}\log\left( \frac{4^6n^5}{\epsilon^6}\log(2n)^2\right)\\
&\approx \frac{40n}{\epsilon}\log\left( \frac{(4n)^6}{\epsilon^6}\right)\\
&= \frac{240n}{\epsilon}\log\left( \frac{4n}{\epsilon}\right).
\end{split}
\end{equation*}

The number of iterations performed by the power method and thus, the number of matrix-vector multiplications at each $t$ is then bounded by $240 \frac{n}{\epsilon}\log\left(\frac{4n}{\epsilon}\right)$. Furthermore, at iteration $t$ of Algorithm~\ref{Algo1}, we keep track of the sample $z$ and $\mathcal{A}(X)$ which requires storage of $2n$ numbers. Furthermore, the working memory of the power method is $n$ numbers. This leads to a total working memory of at most $3n$ numbers.
\end{proof}

\begin{remark}
If $C\succeq 0$, but not diagonally dominant, then the suboptimality bound~\eqref{eqn:MCoptimality} holds for $M = (n+2)\frac{\log(2d)}{\epsilon}$ and $\beta = (n+2)\textup{Tr}(C)$.
However, due to the dependence of the values of parameters $\beta$ and $M$ on $n$ in this case, it takes $\mathcal{O}\left(\frac{\log(2n)n^4}{\epsilon^2}\right)$ iterations within Algorithm~\ref{Algo1} to achieve the stated guarantee.
\end{remark}

%% file: ApplicationsSIMODS.tex
\section{From Gaussian to Randomized Extreme-Point Sampling}\label{Sec:Applications}
Until now, we have focused on using Gaussian random vectors to represent PSD matrix decision variables in low memory, and showed how to modify the Frank-Wolfe algorithm to track these samples. In this section, we discuss a more flexible approach to sample-based representations of decision variables.

Consider the problem
\begin{equation}\label{prob:GenSDP}
\max_{x \in \mathcal{S}}\ \ g(\mathcal{B}(x)),
\end{equation}
where $g$ is a smooth concave function, $\mathcal{B}$ is a linear map, and  $\mathcal{S}$ is a compact, convex set.

If $\mathcal{S}$ is the set of trace-constrained PSD matrices, we can apply Algorithm~\ref{Algo1} as seen in Section~\ref{Sec:Methodology}. When the decision variable is not a PSD matrix, it is not immediately clear whether there is a natural analogue of the Gaussian sampling idea from Section~\ref{sec:GS}. One way to proceed, is to think of $\mathcal{S}$ as the set of expectations of random variables supported on the extreme points of $\mathcal{S}$. The analogue of the Gaussian sampling idea is to construct a Markov chain on the extreme points of the feasible region so that its expectation converges to an optimal solution of Problem~\eqref{prob:GenSDP}. The updates in the Frank-Wolfe algorithm at each iteration are generated as optimal solutions to linear optimization problem over a convex set and so can be taken to be extreme points of that set. This idea opens up the possibility of developing algorithms for solving Problem~\eqref{prob:GenSDP} that require low working memory by modifying Frank-Wolfe, as long as certain conditions on the feasible set $\mathcal{S}$ are satisfied.

\paragraph{Randomized extreme-point sampling}
The basic idea of randomized extreme-point sampling of~\eqref{eqn:randomizedsampling} is to modify Algorithm~\ref{algo:FWapprox} so its state is a random extreme point $z_t$ with expectation $x_t$. To do this, at iteration $t$, we update the random extreme point via

\begin{equation}\label{eqn:randomizedsampling}
z_{t+1}=
    \begin{cases}
      z_{t} & \textup{with probability}\ 1-\gamma_t \\
      h_t & \textup{with probability}\ \gamma_t,
    \end{cases}
\end{equation}
where $h_t$ is an update direction that is an extreme point of $\mathcal{S}$. Note that this update direction is computed as in Algorithm~\ref{algo:FWapprox} and is deterministic since it depends on the variable $v_t = \mathcal{B}(x_t)$ that we track along with the sample $z_t$.

The expected value of $z_{t+1}$ is $\mathbb{E}[z_{t+1}] = (1-\gamma_t)\mathbb{E}[z_t] + \gamma_t h_t$. By induction, it follows that at every iteration $t$, $\mathbb{E}[z_t] = x_t$ and $\mathbb{E}[z_{t+1}] = x_{t+1} = (1-\gamma_t)x_t + \gamma_t h_t$. Since this is equivalent to the update rule of Algorithm~\ref{algo:FWapprox} in expectation, the convergence rate given in Theorem~\ref{thm:FWapproxConv} also holds for $\mathbb{E}[z_t]$. Thus, replacing the solution $x_t$ with a random sample $z_t$, we get Frank-Wolfe with randomized extreme-point sampling, whose outline is given in Algorithm~\ref{algo:FWRandom}.

\begin{algorithm}[ht] 
\SetAlgoLined
\SetKwInOut{Input}{Input}\SetKwInOut{Output}{Output}
\DontPrintSemicolon

\Input{Problem~\eqref{prob:GenSDP}}
\Output{A sample $z$ such that $\mathbb{E}[z] = \widehat{x}_{\epsilon}$, where $\widehat{x}_{\epsilon}$ is an $\epsilon$-optimal solution of~\eqref{prob:GenSDP}}
\BlankLine

\SetKwFunction{FWR}{FWExtremePoint}
\SetKwProg{Fn}{Function}{:}{\KwRet{$z_t$}}

 \Fn{\FWR}{
 Initialize $x_0\in \mathcal{S}$, $v_0 = \mathcal{B}(x_0)$ and set $z_0$ to be a random extreme point with $\mathbb{E}[z_0] = X_0$\;
 Set $t=0$, $\gamma_t = \frac{2}{t+2}$\;
 \While{stopping criteria is not satisfied}{
 Using \texttt{LMO}, compute the update direction $h_t$, and $q_t = \mathcal{B}(h_t)$ \;
 Update $z_t$ using~\eqref{eqn:randomizedsampling}\;
 Set $v_{t+1} \leftarrow (1-\gamma)v_t + \gamma q_t$\;
 $t \leftarrow t+1$, $\gamma_t \leftarrow \frac{2}{t+2}$\;
 }
 }
\caption{\label{algo:FWRandom} (\texttt{FWExtremePoint}) Outline of Frank-Wolfe Algorithm with Randomized Extreme-Point Sampling}
\end{algorithm}

In order to implement Algorithm~\ref{algo:FWRandom} in low memory, we need the conditions~\ref{condition1} and~\ref{condition2} from Section~\ref{Sec:Prelimnaries} to be satisfied. 
%These conditions state that it must be possible to carry out the computations required to compute the update direction in low memory, and that we should be able to represent the extreme points of the feasible set $\mathcal{S}$ in low memory. 
In the rest of the section, we look at the application of randomized extreme-point sampling to example problems which satisfy these two conditions.
%conditions~\ref{condition1} and~\ref{condition2}. 
First, in Section~\ref{sec:Chordal}, we look at SDPs with rank-1 extreme points, more specifically, SDPs where the decision variable is sparse with respect to a chordal graph, where the working memory is bounded by $\mathcal{O}(d+|V|)$. Next, in Section~\ref{sec:samplingVecVar}, we look at 
the compressive sensing problem, which has a vector decision variable in an $n$-dimensional space.
%problems with vector decision variable in $n$-dimensional space, where the number of nonzero entries in the solution is much smaller than $n$. For sensor selection problem, we show that the working memory can be limited to $\mathcal{O}(k\log(n)+m^2)$, where $k$ is the number of sensors selected and $m$ is the number of measurements. Furthermore, for compressive sensing 
In this case, we are able to show that the working memory of the algorithm is bounded by $\mathcal{O}(m)$, where $m$ is the number of measurements. In the supplementary material, we also discuss the sensor selection problem of Joshi and Boyd~\cite{joshi2008sensor}. We have summarized the memory requirement to store extreme points and the memory used to perform computations for the problems discussed in this section in Table~\ref{table:MemSec5}.

\begin{table}[htbp]
{\footnotesize
\begin{center}
\caption{Summary of working memory of Algorithm~\ref{algo:FWRandom}, and memory used to store an extreme point (related to the condition~\ref{condition2}) and implement the \texttt{LMO} subroutine (related to the condition~\ref{condition1}) for the examples in Section~\ref{Sec:Applications}. Note that $\tilde{\mathcal{O}}$ suppresses the $\log(n)$ memory used to store an integer between $1$ and $n$.}\label{table:MemSec5}
\begin{tabular}{|c| c| c| c|}
\hline
\parbox[c]{3cm}{\centering Problem} & \multicolumn{2}{c|}{Memory used to} & \parbox[c]{3cm}{\centering Working memory of}\\
\cline{2-3}
\parbox[c]{3cm}{\centering } & \parbox[c]{3cm}{\centering store an extreme point} & \parbox[c]{3cm}{\centering perform computations in \texttt{LMO}} & \parbox[c]{3cm}{\centering Algorithm~\ref{algo:FWRandom}}\\
\hline
SDPs with chordal sparsity&$|V|$&$\mathcal{O}(|V|+d)$&$\mathcal{O}(|V|+d)$\\
\hline
%\rule{0pt}{10pt} Sensor selection problem&$\Tilde{\mathcal{O}}(k)$&$\Tilde{\mathcal{O}}(m^2+k)$&$\Tilde{\mathcal{O}}(m^2+k)$\\
%\hline
\rule{0pt}{10pt} Compressive sensing&$\Tilde{\mathcal{O}}(1)$&$\mathcal{O}(m)$&$\mathcal{O}(m)$\\
\hline
\end{tabular}
\end{center}
}
\end{table}

\subsection{Randomized Extreme-Point Sampling for SDPs with Chordal Sparsity}\label{sec:Chordal}
When the feasible region consists of trace constrained PSD matrices, the randomized extreme-point sampling of~\eqref{eqn:randomizedsampling} can be applied to Problem~\eqref{prob:GenSDP}. A key feature of this constraint set is that all of its extreme points have rank zero or one. Here, the condition~\ref{condition2} from Section~\ref{Sec:Prelimnaries} is met and the extreme points require much less memory than the size of the decision variable. 
The additional flexibility of randomized extreme-point sampling means that this technique is also applicable to a larger class of spectrahedra (i.e.,  feasible regions of SDPs). In seeking feasible regions with extreme rays that have a low memory representation, it is natural to consider spectrahedra with only rank-1 extreme points, which  have been classified by Blekherman, Sinn, and Velasco~\cite{blekherman2017sums}. Rather than discuss this class in general, we focus on the case of PSD matrices that are sparse with respect to a chordal graph.

Given a graph $G = (V,E)$, and a $|V|\times |V|$ symmetric matrix $X$, we say that \emph{$X$ is sparse with respect to $G$} if $X_{ij} = 0$ whenever $(i,j)\notin E$ and $i\neq j$. Consider the convex set
\[ \mathcal{S}_G = \{X\in \mathbb{S}^{|V|}\;:\; \textup{Tr}(X)\leq \alpha, \;X \succeq 0,\;\;\textup{$X$ is sparse with respect to $G$}\}.\]
A graph $G$ is said to be \emph{chordal} if every cycle of $G$ of length at least four has a chord. Let $\mathcal{B}(\cdot): \mathbb{S}^{|V|}_+ \rightarrow \mathbb{R}^d$ be a linear map, and let $G$ be a chordal graph, and consider the problem
%
%In this subsection, we consider the problem
\begin{equation}\label{prob:ChordalSDP}
\max_{x \in \mathcal{S}_G}\ \ g(\mathcal{B}(x)).
\end{equation}
%where 
 If $G$ is chordal, then the extreme points $\textrm{ext}(\mathcal{S}_G)$ of $\mathcal{S}_G$ have the following characterization.
\begin{theorem}\label{thm:ChordalEPs}
If $G$ is a chordal graph, then $X\in \textrm{ext}(\mathcal{S}_G)$ 
%is an extreme point of $\mathcal{S}_G$ 
if and only if $X=0$ or $X = uu^T$, where $u\in \mathbb{R}^{n}$, $\|u\|^2_2 = \alpha$ and the indices of nonzero entries in $u$ form a clique of $G$.
\end{theorem}
\begin{proof}[Sketch of proof]
($\Leftarrow$) For any graph $G$, this follows from the fact that any rank-1 element of $\mathcal{S}_G$ must be an extreme point.
($\Rightarrow$) This is a consequence of~\cite[Theorem~2.3]{agler1988positive}, which states that any PSD matrix sparse with respect to a chordal graph decomposes as a sum of PSD matrices, each sparse with respect to some maximal clique of $G$.
\end{proof}

Thus, given an input chordal graph $G$ and its set of maximal cliques, the extreme points of $\mathcal{S}_G$ have rank at most one, with the number of nonzero elements in the rank-1 factorization upper bounded by the size of the largest clique. The memory required to represent each extreme point is then bounded above by the number of vertices in the graph and satisfies the condition~\ref{condition2}. Note that the number of maximal cliques of a chordal graph is bounded above by $|V|$, the number of vertices.

We also need to check that the condition~\ref{condition1} holds for Algorithm~\ref{algo:FWRandom} when applied to Problem~\eqref{prob:ChordalSDP} to ensure that it is a low memory algorithm. Computing the update direction now requires solving one eigenvalue problem for each maximal clique of $G$. These can be solved serially each via the power method so that we get an update direction that is a zero-padded vector representing an extreme point of $\mathcal{S}_G$. The overall memory required by Algorithm~\ref{algo:FWRandom} for these computations is still bounded by the dimension of the codomain of the linear map $\mathcal{B}$ and the size of the largest maximal clique in the chordal graph $G$. As such, the condition~\ref{condition1} is satisfied.

The difference between the feasible region of~\eqref{prob:PrimalSDP} and the feasible region $\mathcal{S}_G$ of Problem~\eqref{prob:ChordalSDP} is the additional $\binom{|V|}{2}-|E|$ linear constraints. In Section~\ref{sec:LintoBounded}, we dealt with such additional constraints by incorporating an associated penalty into the objective function at the expense of infeasibility and increasing the curvature constant of the objective. However, when $G$ is chordal, the extreme points of $\mathcal{S}_G$ have a concise representation given by Theorem~\ref{thm:ChordalEPs}, and satisfy conditions~\ref{condition1} and~\ref{condition2}. Using Algorithm~\ref{algo:FWRandom} eliminates the need to penalize these additional constraints.

If the graph $G$ is not chordal, we can combine the penalization approach and the chordal graph approach as follows. We add extra edges to the graph $G$ to make a chordal graph $\overline{G}$ (known as a chordal cover). Because we have added extra edges to $G$ to get $\overline{G}$, $\mathcal{S}_G \subseteq \mathcal{S}_{\overline{G}}$. The constraints corresponding to $E(\overline{G})\backslash E(G)$ could then be penalized in the objective function using the penalty function defined by equation~\eqref{eqn:penaltyPhiDef}. This gives a problem with a modified objective function and a convex feasible region where the decision variable is sparse with respect to the chordal graph $\overline{G}$. Using Algorithm~\ref{algo:FWRandom} will now generate a near-feasible, near-optimal solution to the problem defined on the graph $G$.

\subsection{Randomized Extreme-Point Sampling for Compressive Sensing}\label{sec:samplingVecVar}
The applications considered until now have a PSD matrix decision variable. In this subsection, we illustrate the randomized extreme-point strategy for the compressive sensing problem, in which the decision variable is a nonnegative vector. We discuss assumptions on the problem data under which randomized extreme-point sampling can be implemented in low memory.

Compressive sensing is used to reconstruct a sparse signal $x\in \mathbb{R}^n$ from a set of noisy linear measurements $w = Ax + \textup{noise} \in \mathbb{R}^m$ with $m \ll n$~\cite{orovic2016compressive}. If, in addition, $x$ is nonnegative, a standard convex formulation of the problem is
\begin{equation}\label{prob:CompSensing}
\min_x\ \ \frac{1}{2} \| Ax - w \|^2_2\ \ \ \textup{subject to}\ \ \ \begin{cases}
&\| x\|_1 \leq \alpha\\
&x \geq 0.
\end{cases}
\end{equation}

The extreme points of the feasible region 
%$\mathcal{S} = \{x\in \mathbb{R}^n: \| x\|_1 \leq \alpha,x \geq 0\}$ of Problem~\eqref{prob:CompSensing} 
are the origin and vectors that have a single nonzero element with value equal to $\alpha$. These can be represented as a singleton containing the index of the nonzero element, requiring $\log(n)$ bits of storage.

Let $\mathcal{B}(\cdot): \mathbb{R}^n \rightarrow \mathbb{R}^m$ be defined as $\mathcal{B}(x) = Ax = v$, so that the objective function of Problem~\eqref{prob:CompSensing} is of the form, $g(\mathcal{B}(x)) = \frac{1}{2}\| \mathcal{B}(x) - w\|^2_2$. The update direction is computed as
\begin{equation} \label{prob:SubProbCS}
%h^{(t)} = \argmax_{\substack{\| d\|_1\leq \alpha,\\ d\geq 0}}\ \langle \nabla g(\mathcal{B}(x^{(t)})), d\rangle.
h^{(t)} = \argmax_{\| d\|_1\leq \alpha, \;d\geq 0}\ \langle \nabla g(\mathcal{B}(x^{(t)})), d\rangle.
\end{equation}
If $i^\star \in \argmax_i\,\nabla g(\mathcal{B}(x^{(t)}))_i$, then an optimal solution of Problem~\eqref{prob:SubProbCS} is a vector with a single nonzero element indexed by $i^\star$ and whose value is equal to $\alpha$ (or the zero vector if $\nabla g(\mathcal{B}(x^{(t)})) \leq 0$).
Since computing $i^\star$ is equivalent to finding the largest element in $\nabla g(\mathcal{B}(x^{(t)}))$, the \ref{condition1} property holds as long as we can generate the columns of $A$ serially without explicitly storing them. This is possible, e.g., if $A$ were a partial Fourier matrix. The working memory of Algorithm~\ref{algo:FWRandom} is effectively restricted to $\mathcal{O}(m)$ numbers required to store $v$ and an extreme point.

\paragraph{Recovering a signal with $k$ nonzero elements}
The randomized extreme-point sampling algorithm returns a random index $i\in \{1,\dotsc,n\}$ that is a sample from the distribution defined by normalizing a near-optimal point for Problem~\eqref{prob:CompSensing}. In Section~\ref{sec:PostProc}, we briefly discuss how to recover a $k$-sparse approximation of the optimal solution to Problem~\eqref{prob:CompSensing}. 
%The main idea is to run the randomized extreme-point sampling algorithm multiple times and use a streaming algorithm to find the $k$ most frequently occurring indices in the resulting stream.

\section{Post-Processing of Samples}\label{sec:PostProc}
In the previous sections, we developed algorithms that generate samples that represent an $\epsilon$-optimal solution of a convex optimization problem of the form~\eqref{prob:GenSDP}. When Gaussian sampling is used with Algorithm~\ref{Algo1}, the resulting zero-mean Gaussian samples have \textit{covariance} that represents an $\epsilon$-optimal solution of the input problem. Whereas, when randomized extreme-point sampling (Algorithm~\ref{algo:FWRandom}) is used, the output is a sample whose \textit{expected value} represents an $\epsilon$-optimal solution to the problem. In this section we briefly discuss further processing that can be performed on these samples to generate other memory-efficient (such as sparse or low rank) approximations of that near-optimal solution. The general approach will be to make use of various streaming algorithms.

%For instance, given a sequence of i.i.d.\ samples $z_i \sim \mathcal{N}(0,X)$ generated by running Algorithm~\ref{Algo1} for problems with a unique optimal solution, we would like to estimate either $X$ itself, or its \textit{low memory approximation}. Low memory approximations of interest include rank-$k$ approximations, $\Omega$-sparse approximations or the top $k$ principal components. When the samples are generated sequentially, processed and discarded, such covariance estimation or approximation problems can potentially be solved in a memory-efficient way using streaming algorithms. When the decision variable is (up to scaling) a probability vector, as in compressive sensing (see Section~\ref{Sec:Applications}), Algorithm~\ref{algo:FWRandom} gives a stream of samples from that probability distribution. A natural post-processing task is to find a $k$-sparse approximation of the underlying decision variable. In the rest of the section, we briefly discuss how to post-process sampled solutions to construct such approximations of the solution of the underlying convex program.

\paragraph{Finding low-rank approximation of covariance matrix} Given a sequence of i.i.d.\ samples $z_1,z_2,\dotsc, z_N \sim \mathcal{N}(0,X)$, such as those generated by Algorithm~\ref{Algo1}, 
we can use the method proposed by  Tropp et al.~\cite{tropp2017fixed} to obtain a rank-$r$ approximation of the sample covariance matrix $X_N = \frac{1}{N}\sum_{i=1}^N z_iz_i^T$. The method involves generating and updating a linear sketch $Y_N=X_N\Omega$ of the sample covariance matrix, where $\Omega\in \mathbb{R}^{n\times k}$ is a fixed matrix with i.i.d.\ standard Gaussian entries, and $r\leq k\leq n$. Given a new sample $z_{N+1}\sim \mathcal{N}(0,X)$, the sketch is updated via $Y_{N+1} = \frac{N}{N+1} Y_N + \frac{1}{N+1} z(z^T\Omega)$. Note that this sketch uses $\Theta(kn)$ memory and the computational cost of updating the sketch is $\Theta(kn)$. Furthermore, using \cite[Algorithm~3]{tropp2017fixed} it is possible to reconstruct a rank-$r$ approximation $\widehat{X}_N$ of the sample covariance matrix $X_N$ from $Y_N$. In particular, if $k \sim \Theta (r/\epsilon)$, it is possible to generate $\widehat{X}_N$ such that $\mathbb{E}\|X_N-\widehat{X}_N\|_1 \leq (1+\epsilon)\|X_N-[X_N]_r\|_1$, where $[X_N]_r$ is the best rank-$r$ approximation of the sample covariance and $\|\cdot\|_1$ is the Schatten-1 norm. By choosing sufficiently many samples $N$, we can ensure that $\mathbb{E}\left[ \| X- \widehat{X}_N\|_1\right] \leq (1+2\epsilon)\| X- [X]_r\|_1$, where $X$ is the population covariance of the samples and $[X]_r$ is its best rank-$r$ approximation. This means that using $\mathcal{O}(nr/\epsilon)$ memory, we can post-process the sampled output of Algorithm~\ref{Algo1} to obtain a rank-$r$ approximation of a near-optimal solution of~\eqref{prob:PrimalSDP}.

%\paragraph{Finding top eigenspace using streaming PCA} Given a sequence $z_1, z_2, \dotsc, z_N \sim \mathcal{N}(0,X)$ of i.i.d.\ samples with bounded norm, we may want to estimate the eigenspace corresponding to the $k$ largest eigenvectors. Unlike estimating a rank-$k$ approximation, we require a nonzero eigengap between the eigenvalues $\lambda_k$ and $\lambda_{k+1}$ of $X$ for this problem to be well defined. This problem is a variant of streaming PCA which computes a $k$-dimensional subspace that best contains $n$-dimensional samples presented sequentially. The SPCA algorithm given by Oja and Karhunen~\cite{oja1985stochastic} keeps track of an $n\times k$ matrix with orthonormal columns, say $Q_N$, at each step $N$. When a new sample $z_{N+1} \sim \mathcal{N}(0,X)$ with bounded norm is received, it performs the update $Q_{N+1} = \textup{QR}(Q_N + \gamma_N z_{N+1}(z_{N+1}^TQ_N))$ where $\textup{QR}(\cdot)$ computes the QR decomposition so that the matrix $Q_{N+1}$ has orthonormal columns and $\gamma_N > 0$ is a step size parameter. Furthermore, using SPCA, it is possible to compute the top $k$ eigenspace approximately with prespecified error by choosing sufficiently many samples (see~\cite[Theorem~1]{li2016rivalry}). Since the samples are processed sequentially and discarded, the streaming PCA model requires at most $\mathcal{O}(kn)$ memory to compute and store the top $k$ eigenspace of the covariance matrix.

\paragraph{Finding a signal vector with $k$ nonzero elements}
The output of Algorithm~\ref{algo:FWRandom}, when applied to compressive sensing, is a stream of data where each data point represents an index of a single nonzero element in the signal. As seen in Section~\ref{sec:samplingVecVar}, the goal of compressive sensing is to recover a signal with $k$ nonzero elements. By recovering the frequency of $k$ most frequently occurring indices in this data stream, we recover a signal with $k$ nonzero elements that satisfies the constraints (up to scaling) for compressive sensing. The counter based technique proposed by  Metwally, Agrawal, and El Abbadi~\cite{metwally2005efficient} is a memory-efficient way to approximately compute these elements by only keeping track of the counts of occurrence of few elements, say $K \geq k$, at a time. 
%When a new element $z_i$ arrives, the count of the element is updated by one if it is being tracked at that time. Otherwise, it displaces the element with the lowest count in the list and the count of the new element is set to the count value of the element it displaces plus one. The data structure is a linked list with the elements stored in the decreasing order of their frequency. 
By setting $K = \frac{1}{\epsilon f_k}$, where $f_k$ is the frequency of the $k$-th most common element, it is possible to find the top $k$ frequently occurring elements such that the frequency of each element is at least $(1-\epsilon)f_k$~\cite[Theorem~6]{metwally2005efficient}.

%% file: ConclusionSIMODS.tex
\section{Discussion}\label{Sec:Conclusion}

\paragraph{Comparison with Yurtsever et al.~\cite{yurtsever2019scalable}}
In Section~\ref{Sec:Methodology}, we saw that it is possible to generate and store samples of a near-feasible, near-optimal solution to~$\eqref{prob:SDPlin}$ using Algorithm~\ref{Algo1} with working memory that is independent of the approximation parameter~$\epsilon$ and limited to $\mathcal{O}(d+n)$. This memory requirement differs from the $\mathcal{O}(d+rn/\zeta)$ working memory  used by the algorithm given by Yurtsever et al.~\cite{yurtsever2019scalable} to achieve  $\|X-\widehat{X}\|_{\star} \leq (1+\zeta)\|X-[X]_r\|_{\star}$). Here, $\widehat{X}$ is the rank-$r$ approximation produced by their algorithm and $[X]_r$ is the best rank-$r$ approximation of the solution $X$. The difference arises because we only aim to provide a sampled representation of the approximate solution rather than generate a near best rank $r$ approximation of the solution, or recover the exact solution matrix. In the special case of \textsc{MaxCut}, our sample-based representation is sufficient to implement an $(1-\epsilon)\alpha_{GW}$ approximation algorithm based on the Goemans-Williamson rounding scheme using memory linear in~$n$ and independent of~$\epsilon$. Yurtsever et al.~\cite{yurtsever2019scalable} have performed numerical experiments which show that their method is capable of handling~\eqref{prob:relaxedMC} with $n\approx 8\cdot 10^6$ on a computer with 16 GB RAM. Our preliminary numerical results (see Section~\ref{sec:ComputationalResults}) are perhaps less promising in terms of practical convergence rate, but are competitive in terms of memory usage.

\paragraph{Alternative algorithms with Gaussian sampling} The Frank-Wolfe algorithm is well suited for Gaussian and extreme-point sampling when the extreme points have low rank. It is interesting to consider which other algorithms can be modified to track a sampled representation of the decision variable rather than the variable itself. For example, \textsc{MaxCut} algorithm given by Klein and Lu~\cite{klein1996efficient} generates a rank-1 update at each iteration and its output is used to generate a factorization of an approximate solution to~\eqref{prob:relaxedMC}. The structure of the updates and the computations required in their algorithm are structurally similar enough to our approach that their matrix iterates can be systematically replaced with Gaussian samples.

Another method where this could be done is the Matrix Multiplicative Weights (MMW) method, where the update to the variable takes the form $x_{t+1} \leftarrow x_t\exp(c(x))$ with $c(x)$ being a feedback function from the previous iterates. In the case of SDPs, this method requires computing the matrix exponential to generate the updates. Carmon et al.~\cite{carmon2019rank} provide an algorithm~\cite[Algorithm~1]{carmon2019rank} for solving~\eqref{prob:SDPlin} using a variation on the MMW method which relieves the computational burden of generating the matrix exponential. They do so by restricting the update to be a rank-1 sketch which is the result of multiplying the matrix exponential with a random vector drawn from a standard Gaussian distribution. This \mbox{rank-1} sketch is computed using the Lanczos algorithm without actually generating the matrix exponential. Again we expect that it should be possible to use the idea of randomized extreme-point sampling to get linear working memory while implementing this algorithm.

\paragraph{Using Gaussian sampling for other rounding schemes} Other approximation algorithms involve solving~\eqref{prob:relaxedMC} (with appropriate cost matrix $C$) and then rounding Gaussian samples with covariance given by the SDP solution, such as  \textsc{Max-2SAT} \cite{GW} and the maximization of indefinite binary quadratic forms \cite{charikar2004maximizing, megretski2001relaxations}. Using our approach these can also be solved in $\mathcal{O}(n)$ working memory. It would be interesting to investigate which other rounding schemes can be implemented in a memory-efficient way by modifying our approach.

\subsection{Preliminary Computational Results}\label{sec:ComputationalResults}
We conclude with some preliminary computational experiments for \textsc{MaxCut}. The algorithms we propose are simple to implement, and offer scope for modification and improvement. Our aim is to illustrate this simplicity and identify possible areas for future algorithmic developments. The input parameter values for~\eqref{prob:relaxedMC} were set as $d = n$, $\alpha = n$, $\omega = 1$, $\epsilon = 0.1$ $\beta = 4\textup{Tr}(C)$ and $M = 4\frac{\log(2d)}{\epsilon}$. We use unweighted graphs from \textsc{Gset} dataset with size varying between $n=800$ and $n=7000$. The computations were performed using MATLAB R2018b on a machine with 8GB RAM and 4 cores. The peak memory requirement was noted using the \texttt{profiler} command in MATLAB.

\paragraph{Comparison of storage cost} To compare the memory cost, we solved~\eqref{prob:relaxedMC} for graphs from \textsc{Gset} using Algorithm~\ref{Algo1}, \textsc{SketchyCGAL}\cite[Algorithm 6.1]{yurtsever2019scalable} (with $R=10$) and the following solvers: (i) SeDuMi \cite{sturm1999using}, (ii) SDPT3 \cite{toh1999sdpt3}, (iii) SDPNAL+ \cite{yang2015sdpnal}. In case of \textsc{SketchyCGAL}, we used the default parameter values with the size of the sketch, $R$, set to 10. For each of the remaining solvers (i)-(iii), the tolerance level was set at $10^{-3}$. For Algorithm~\ref{Algo1}, we terminated the algorithm after at most five hours of runtime if an $\epsilon$-optimal solution was not generated. The comparison of memory required is shown in Figure~\ref{fig:MemoryPlot}. In each case, while  Algorithm~\ref{Algo1} converged more slowly, it used less memory than the other solvers. The key observation during the implementation was that the working memory at any time during the running of Algorithm~\ref{Algo1} was linearly proportional to the size of the problem, $n$ and the number of edges, $m$. For slightly denser graphs, the value of $m$ dominated the storage cost, leading to similar memory requirement for Algorithm~\ref{Algo1} and \textsc{SketchyCGAL}. However, with increasing problem size, when the graphs are relatively sparser, Algorithm~\ref{Algo1} required slightly less memory. We postulate that this is because the storage cost for \textsc{SketchyCGAL} is proportional to $Rn$ as opposed to $n$ in case of Algorithm~\ref{Algo1}.

\begin{figure}[htb]
\centering
\includegraphics[width=.75\textwidth]{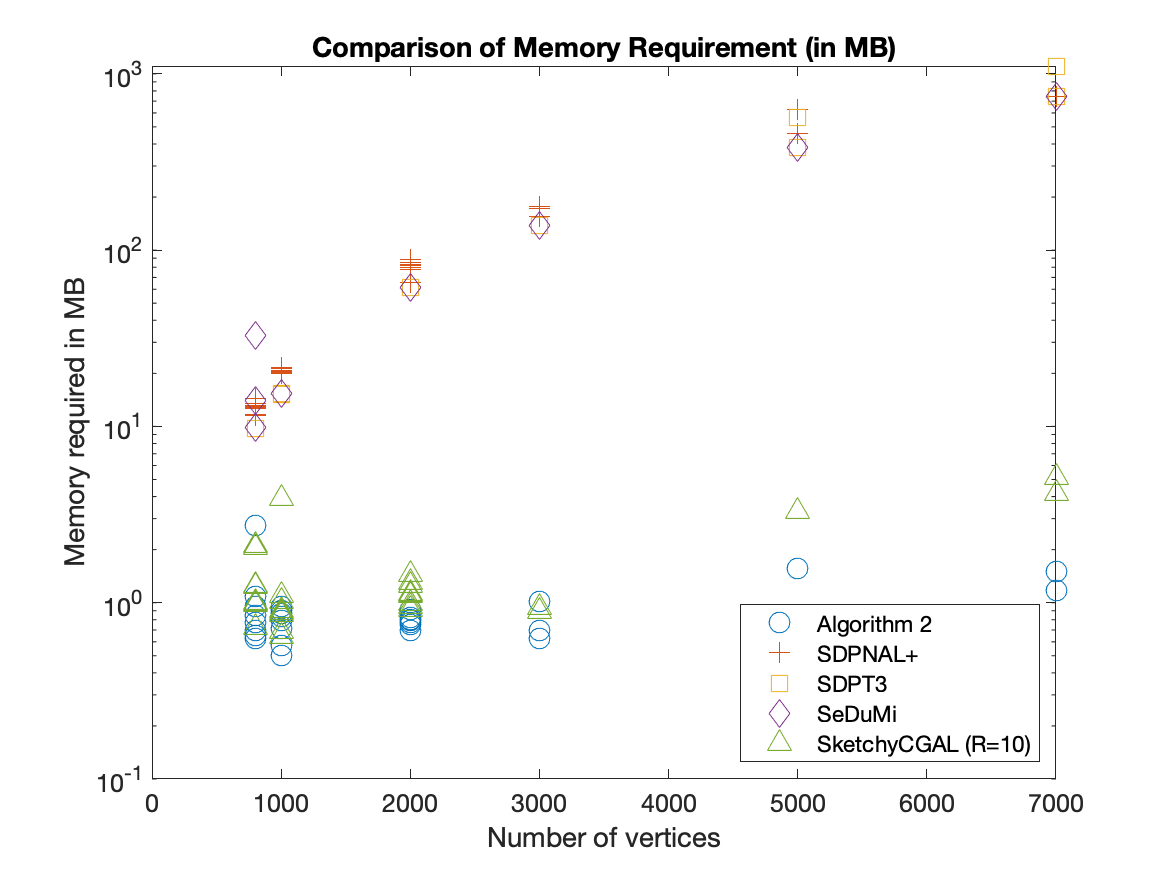}
 \caption{Comparison of memory used (in MB) to solve~\eqref{prob:relaxedMC}. Each point represents a graph from the \textsc{Gset} dataset (see Table~\ref{table:ResultAlgo2GSet} for the list of graphs up to size $n=3000$ and their results).}\label{fig:MemoryPlot}
\end{figure}

\begin{figure}[tbhp]
\centering
\subfloat[G1 ($n=800$)]{\label{fig:1a}\includegraphics[width = 3.05in]{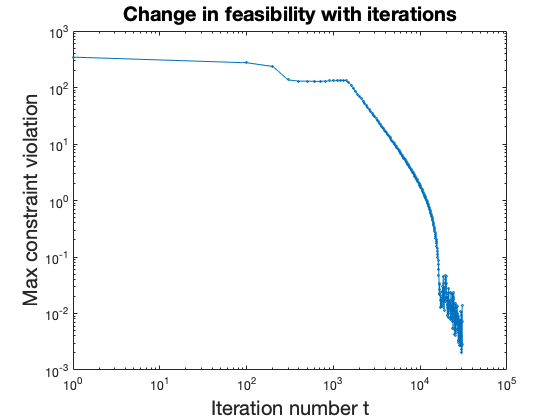}}
\subfloat[G22 ($n=2000$)]{\label{fig:1b}\includegraphics[width = 3.05in]{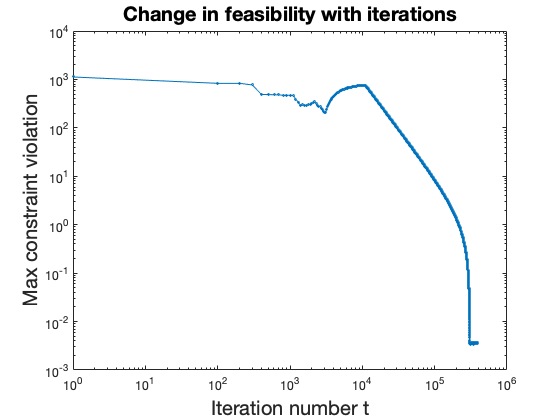}}
\caption{Plot of $\log(\parallel \mathcal{A}(X_t) - b\parallel_{\infty})$ vs $\log(t)$.}
\label{fig:ConstrViol}
\end{figure}

\paragraph{Analysis of Algorithm~\ref{Algo1}} While using Algorithm~\ref{Algo1} to solve~\eqref{prob:relaxedMC}, we implemented the algorithm exactly as given in Section~\ref{Sec:Methodology} without any enhancements with the exception of using \texttt{eigs} command in MATLAB instead of the power method to compute eigenvectors in the \texttt{LMO} subroutine at each iteration. The result of using Algorithm~\ref{Algo1} to compute an $\epsilon$-approximate solution for a subset of instances from \textsc{Gset} dataset is given in Table~\ref{table:ResultAlgo2GSet}. The suboptimality and infeasibility bounds given in Lemma~\ref{lemma:SubOptFeasMaxCut} state that an $\epsilon$-optimal solution, $\widehat{X}_{\epsilon}$ satisfies $\frac{\langle C,X^{\star}_{SDP}\rangle - \langle C, \widehat{X}_{\epsilon}\rangle}{\langle C,X^{\star}_{SDP}\rangle}\leq \epsilon$ and $\| \mathcal{A}(\widehat{X}_{\epsilon}) - \|_{\infty}\leq \epsilon$. From Table~\ref{table:ResultAlgo2GSet}, we note that these bounds were satisfied and the number of iterations required for convergence was also within the bounds given in Theorem~\ref{thm:MaxCutResult}. Table~\ref{table:ResultAlgo2GSet} shows the value of cut that was generated by rounding a single sample, i.e., it represents the value $w^TCw$ rather than $\mathbb{E}[w^TCw]$. This generated cut value is not as good as the one generated by \textsc{SketchyCGAL} (see Table~\ref{table:ResultAlgo2GSet}). However, we note that for every problem instance, the cut value, $w^TCw$, was greater than $\frac{\alpha(1-3\epsilon)}{1-\epsilon}\langle C, \widehat{X}_{\epsilon}\rangle$, which implied $w^TCw \geq \alpha(1-3\epsilon)\langle C, X^{\star}_{SDP}\rangle \geq \alpha(1-3\epsilon)\textup{opt}$.

We also tracked the change in infeasibility, which also determines an upper bound on sub-optimality of the solution, with each iteration. The plot of $\log(\| \mathcal{A}(X) - b\|_{\infty})$ vs $\log(t)$ (iteration number) for problem instances \textsc{G1} and \textsc{G22} are shown in Figure~\ref{fig:ConstrViol}. Comparing the two plots, we see that the rate of the change in error observed was similar for problems with different sizes. For both problems, during the initial phase, the change in infeasibility is small. However, after a fixed percentage of total iterations, there is a steady decrease in the error and finally, it converges to a value about 10 times smaller than $\epsilon = 0.1$ in both cases.

\begin{table}[htbp]
{\footnotesize
\begin{center}
\caption{Computational result of using Algorithm~\ref{Algo1} to compute $0.1$-approximate solution to \textsc{MaxCut} for graphs from the \textsc{Gset} dataset. The optimal value, $\langle C, X^{\star}_{SDP}\rangle$, of~\eqref{prob:relaxedMC} is generated using SDPT3 \cite{toh1999sdpt3} where the tolerance level is set to a default value of 1e-8, $\textup{ogap} = \frac{\langle C,X^{\star}_{SDP}\rangle - \langle C, \widehat{X}_{\epsilon}\rangle}{\langle C,X^{\star}_{SDP}\rangle}$, \texttt{CVG} is the value of cut generated by Algorithm~\ref{Algo1} and \texttt{CVS} is the value of cut generated by \textsc{SketchyCGAL}.}\label{table:ResultAlgo2GSet}
\begin{tabular}{|c| c| c| c| H c| c|c|c|}
\hline
\parbox[c]{1.3cm}{\raggedright Dataset ($n$)} & \parbox[c]{1cm}{\raggedright \texttt{CVS}} & \parbox[c]{1cm}{\raggedright \texttt{CVG}} & \parbox[c]{1.4cm}{\raggedright \# Iterations}  & \parbox[c]{1cm}{\raggedright Time (in secs)} & \parbox[c]{1.6cm}{\raggedright $\| \mathcal{A}(\widehat{X}_{\epsilon}) -b\|_{\infty}$} & \parbox[c]{1.3cm}{$\langle C, \widehat{X}_{\epsilon}\rangle$} & \parbox[c]{1.8cm}{$\langle C, X^{\star}_{SDP}\rangle$} & \parbox[c]{1.1cm}{$\textup{ogap}$}\\
\hline
G1 (800)&11410&10903&25285&299.56&0.005&11482.337&12083&0.05\\
\hline
G2 (800)&11378&10808&25027&292.65&0.007&11516.227&12089&0.047\\
\hline
G3 (800)&11376&10734&53597&890.55&0.002&11164.577&12084&0.076\\
\hline
G4 (800)&11395&10834&23282&210.28&0.006&11581.868&12111&0.044\\
\hline
G5 (800)&11386&10942&25155&245.60&0.006&11520.603&12100&0.048\\
\hline
G14 (800)&2933&2695&186125&1757.23&0.026&2888.858&3191.6&0.095\\
\hline
G15 (800)&2946&2575&213059&1859.53&0.028&2871.221&3171.6&0.095\\
\hline
G16 (800)&2929&2702&180269&1641.31&0.025&2882.307&3175&0.092\\
\hline
G17 (800)&2941&2699&195772&1882.41&0.027&2879.435&3171.3&0.92\\
\hline
G22 (2000)&12919&12042&359679&12394.25&0.004&12902.273&14136&0.087\\
\hline
G23 (2000)&12963&11953&127076&4069.39&0.003&13035.692&14142&0.078\\
\hline
G24 (2000)&12888&11869&114487&2323.00&0.004&13039.235&14141&0.078\\
\hline
G25 (2000)&12894&12059&75340&1565.16&0.004&13111.383&14144&0.073\\
\hline
G26 (2000)&12918&11889&95181&1912.94&0.004&13057.627&14133&0.076\\
\hline
G35 (2000)&7365&6465&600032&7175.32&0.029&7260.209&8014.7&0.094\\
\hline
G36 (2000)&7381&6523&730590&12110.95&0.034&7286.916&8006&0.09\\
\hline
G37 (2000)&7373&6621&805089&13827.92&0.032&7276.509&8018.6&0.093\\
\hline
G43 (1000)&6512&6300&27416&342.83&0.004&6626.791&7032.2&0.058\\
\hline
G44 (1000)&6438&6170&25333&322.24&0.004&6629.456&7027.9&0.057\\
\hline
G45 (1000)&6470&6209&28319&352.49&0.003&6608.589&7024.8&0.059\\
\hline
G46 (1000)&6437&6050&25574&316.27&0.004&6650.921&7029.9&0.054\\
\hline
G47 (1000)&6426&6186&34344&301.41&0.003&6584.912&7036.7&0.064\\
\hline
G48 (3000)&6000&5284&14145&131.13&0.039&5849.946&6000&0.025\\
\hline
G49 (3000)&6000&5414&18803&183.60&0.018&5806.924&6000&0.032\\
\hline
G50 (3000)&5858&5502&19394&192.65&0.028&5782.197&5988.2&0.034\\
\hline
G51 (1000)&3716&3413&327568&11117.69&0.028&3714.973&4006.3&0.073\\
\hline
G52 (1000)&3701&3383&139494&1984.68&0.029&3716.020&4009.6&0.073\\
\hline
G53 (1000)&3719&3456&232629&4019.05&0.026&3712.973&4009.7&0.074\\
\hline
G54 (1000)&3710&3518&270915&5235.46&0.025&3707.351&4006.2&0.075\\
\hline
\end{tabular}
\end{center}
}
\end{table}

The slow initial convergence shown in Figure~\ref{fig:ConstrViol} indicates that there is room for improvement in the design of the algorithm. We conjecture that this slow initial convergence is due to the approach taken to penalize the constraints and identify this as a natural direction for future algorithmic work. Currently, each iteration also requires computing the leading eigenvector. This could be improved with warm start if we know the approximate subspace in which the eigenvector lies which might become clearer as the algorithm reaches the near-feasible, near-optimal solution to the input problem. These improvements can potentially lead to a more practical low memory method, based on the ideas presented in this paper.

\appendix

%% file: AppendixSIMODS.tex
\section{Proof of Lemma~\ref{lemma:SubOptFeas}}\label{appendix:OptFeasLin}

\begin{proof}
There are three inequalities to prove.

\paragraph{Lower bound on the objective function value, $\langle C, \widehat{X}_{\epsilon} \rangle$}
Let $X^\star_{FW}$ be an optimal solution to~\eqref{prob:LogPenalty}. After the stopping criteria of Algorithm~\ref{Algo1} is satisfied, the following holds:
\begin{equation}\label{eqn:SubOptimalSolPenaltyProb}
g(\mathcal{B}(\widehat{X}_{\epsilon})) \geq g(\mathcal{B}(X^{\star}_{FW})) - \epsilon \geq g(\mathcal{B}(X^{\star}_{SDP})) - \epsilon
\end{equation}
since $X^{\star}_{SDP}$ is feasible for~\eqref{prob:LogPenalty}. Thus,
\begin{equation}
\langle C, \widehat{X}_{\epsilon} \rangle - \beta \phi_M (\mathcal{A}(\widehat{X}_{\epsilon})-b) \geq \langle C, X^{\star}_{SDP}\rangle - \beta \phi_M (\mathcal{A}(X^{\star}_{SDP})-b) - \epsilon.
\end{equation}
The lower bound in~\eqref{eqn:OptBound} follows since $\phi_M(\mathcal{A}(X^{\star}_{SDP})-b) \leq \phi_M(\mathcal{A}(\widehat{X}_{\epsilon})-b)$.

\paragraph{Upper bound on the objective function value, $\langle C, \widehat{X}_{\epsilon} \rangle$} 
The Lagrangian of~\eqref{prob:SDPlin} is defined as
\begin{equation*}
L(X,y) = \langle C,X\rangle - y^T(\mathcal{A}(X) - b).
\end{equation*}

For a primal-dual optimal pair, ($X^{\star}_{SDP},y^{\star}_{SDP}$) and any $X\succeq 0$, the following holds,
\begin{equation}\label{eqn:LagrangianIneq}
L(X,y^{\star}_{SDP}) \leq L(X^{\star}_{SDP},y^{\star}_{SDP}).
\end{equation}

Since $\widehat{X}_{\epsilon}\succeq 0$, from~\eqref{eqn:LagrangianIneq}, we can write
\begin{equation}
\begin{split}
\langle C, \widehat{X}_{\epsilon} \rangle - y^{\star T}_{SDP}(\mathcal{A}(\widehat{X}_{\epsilon}) - b) &\leq \langle C, X^{\star}_{SDP}\rangle - y^{\star T}_{SDP}(\mathcal{A}(X^{\star}_{SDP}) - b)\\
&= \langle C, X^{\star}_{SDP}\rangle.
\end{split}
\end{equation}

The upper bound on $\langle C, \widehat{X}_{\epsilon} \rangle$ can be written as,
\begin{align}
\langle C, \widehat{X}_{\epsilon} \rangle &\leq \langle C, X^{\star}_{SDP}\rangle + y^{\star T}_{SDP}(\mathcal{A}(\widehat{X}_{\epsilon}) - b)\\
&\leq \langle C, X^{\star}_{SDP}\rangle + \| y^{\star}_{SDP} \|_1 \| \mathcal{A}(\widehat{X}_{\epsilon}) - b\|_{\infty}. \label{eqn:upperbound}
\end{align}

\paragraph{Bound on infeasibility, $\| \mathcal{A}(\widehat{X}_{\epsilon}) - b\|_{\infty}$}

We rewrite~\eqref{eqn:SubOptimalSolPenaltyProb} as,
\begin{equation*}
\begin{split}
\beta \phi_M (\mathcal{A}(\widehat{X}_{\epsilon})-b) &\leq \langle C, \widehat{X}_{\epsilon} \rangle - \langle C, X^{\star}_{SDP}\rangle + \beta \phi_M (\mathcal{A}(X^{\star}_{SDP})-b) + \epsilon\\
&\leq \| y^{\star}_{SDP} \|_1 \| \mathcal{A}(\widehat{X}_{\epsilon}) - b\|_{\infty} + \beta \phi_M (\mathcal{A}(X^{\star}_{SDP})-b) + \epsilon  \ \ \textup{(from~\eqref{eqn:upperbound})}.
\end{split}
\end{equation*}

Now $\phi_M (\mathcal{A}(X^{\star}_{SDP})-b) = \frac{\log(2d)}{M}$ and, from Proposition~\ref{prop:BoundPenalty}, we know that $\| \mathcal{A}(\widehat{X}_{\epsilon}) - b\|_{\infty} \leq \phi_M (\mathcal{A} (\widehat{X}_{\epsilon})-b)$. So,
\begin{equation}
\beta \| \mathcal{A}(\widehat{X}_{\epsilon}) - b\|_{\infty} \leq \| y^{\star}_{SDP} \|_1 \| \mathcal{A}(\widehat{X}_{\epsilon}) - b\|_{\infty} + \beta \frac{\log(2d)}{M} + \epsilon.
\end{equation}

Since $\beta > \| y^{\star}_{SDP} \|_1$ by assumption,
\begin{align*}
&\left(\beta - \| y^{\star}_{SDP} \|_1\right)  \| \mathcal{A}(\widehat{X}_{\epsilon}) - b\|_{\infty} \leq  \beta \frac{\log(2d)}{M} + \epsilon\\
&\Rightarrow\| \mathcal{A}(\widehat{X}_{\epsilon}) - b\|_{\infty} \leq \frac{\beta \frac{\log(2d)}{M} + \epsilon}{\beta - \| y^{\star}_{SDP} \|_1}. \label{eqn:BoundonInfeas}
\end{align*}

So, we get a bound on infeasibility that depends on $\| y^{\star}_{SDP} \|_1$, $M$ and $\beta$.

\paragraph{Revisiting the upper bound on $\langle C, \widehat{X}_{\epsilon} \rangle$}

Substituting the bound on infeasibility into~\eqref{eqn:upperbound} gives
\begin{equation*}
\langle C, \widehat{X}_{\epsilon} \rangle \leq \langle C, X^{\star}_{SDP}\rangle + \| y^{\star}_{SDP} \|_1 \frac{\beta \frac{\log(2d)}{M} + \epsilon}{\beta - \| y^{\star}_{SDP} \|_1}.
\end{equation*}
\end{proof}